\documentclass[12pt]{amsart} %
\usepackage[pdftex]{xcolor,graphicx}
\usepackage{amsmath,amssymb,amsthm,mathrsfs,ascmac,comment, here, url} 
\usepackage{fullpage, caption}
\usepackage[all]{xy}

\usepackage[pagewise, mathlines]{lineno} 
\usepackage{time}
\usepackage[nameinlink]{cleveref}

\usepackage[format=plain, font=normalsize, labelfont=bf
]{caption} 

\numberwithin{equation}{section}
\usepackage{etoolbox}          

\newcommand*\linenomathpatch[1]{%
  \cspreto{#1}{\linenomath}%
  \cspreto{#1*}{\linenomath}%
  \csappto{end#1}{\endlinenomath}%
  \csappto{end#1*}{\endlinenomath}%
}
\newcommand*\linenomathpatchAMS[1]{%
  \cspreto{#1}{\linenomathAMS}%
  \cspreto{#1*}{\linenomathAMS}%
  \csappto{end#1}{\endlinenomath}%
  \csappto{end#1*}{\endlinenomath}%
}

\expandafter\ifx\linenomath\linenomathWithnumbers 
 \let\linenomathAMS\linenomathWithnumbers
 \patchcmd\linenomathAMS{\advance\postdisplaypenalty\linenopenalty}{}{}{}
\else
  \let\linenomathAMS\linenomathNonumbers
\fi

\linenomathpatch{equation}
\linenomathpatchAMS{gather}
\linenomathpatchAMS{multline}
\linenomathpatchAMS{align}
\linenomathpatchAMS{alignat}
\linenomathpatchAMS{flalign}

\makeatletter
\patchcmd{\mmeasure@}{\measuring@true}{
  \measuring@true
  \ifnum-\linenopenaltypar>\interdisplaylinepenalty
    \advance\interdisplaylinepenalty-\linenopenalty
  \fi
  }{}{}
\makeatother

\newcommand{\mf}[1]{{\mathfrak{#1}}}

\newcommand{\bb}[1]{{\mathbb{#1}}}
\newcommand{\mca}[1]{{\mathcal{#1}}}

\newcommand{\To}{\longrightarrow}

\newcommand{\inj}{\hookrightarrow}
\newcommand{\surj}{\twoheadrightarrow}

\newcommand{\congto}{\overset{\cong}{\to}}

\newcommand{\Z}{\bb{Z}}
\newcommand{\Zp}{\bb{Z}_{p}}
\newcommand{\Q}{\bb{Q}}

\newcommand{\C}{\bb{C}}
\newcommand{\Cp}{\bb{C}_{p}}
\newcommand{\F}{\bb{F}}
\newcommand{\p}{\mf{p}}

\newcommand{\ol}{\overline}

\newcommand{\wt}[1]{{\widetilde{#1}}}
\newcommand{\wh}[1]{{\widehat{#1}}}

\newcommand{\ctext}[1]{\raise0.2ex\hbox{\textcircled{\scriptsize{#1}}}}

\newtheorem{thm}{Theorem}[section]
\newtheorem{lem}[thm]{Lemma}
\newtheorem{cor}[thm]{Corollary}
\newtheorem{prop}[thm]{Proposition}  

\theoremstyle{remark}

\theoremstyle{definition}
\newtheorem{defn}[thm]{Definition}
\newtheorem{rem}[thm]{Remark} 

\newtheorem{eg}[thm]{Examples}

\newtheorem{q}[thm]{Question} 

\newtheorem{def/prop}[thm]{Definition/Proposition}

\newtheorem{thmx}{Theorem} 

\newcommand{\red}{}    

\title[The $p$-adic limits of class numbers in $\Z_p$-towers]{
The $p$-adic limits of class numbers in $\Z_p$-towers} 

\author{
Jun Ueki} 
\email{uekijun46@gmail.com}
\address{Department of Mathematics, Faculty of Science, Ochanomizu University; 
2-1-1 Otsuka, Bunkyo-ku, 112-8610, Tokyo, Japan}

\author{
Hyuga Yoshizaki} 
\email{yoshizaki.hyuga@gmail.com}
\address{Department of Mathematics, 
Faculty of Science and Technology, Tokyo University of Science;
2641, Yamazaki, Noda-shi, 278-8510, Chiba, Japan}

\makeatletter
\@namedef{subjclassname@2020}{
\textup{2020} Mathematics Subject Classification}
\makeatother 



\begin{document}

\begin{abstract} 
This article discusses variants of Weber's class number problem in the spirit of arithmetic topology
to connect the results of Sinnott--Kisilevsky and Kionke.   
Let $p$ be a prime number.  
We first prove the $p$-adic convergence of class numbers in a $\Zp$-extension of a global field and a similar result in a $\Zp$-cover of a compact 3-manifold. 
Secondly, we establish an explicit formula for the $p$-adic limit of the $p$-power-th cyclic resultants of a polynomial using roots of unity of orders prime to $p$, the $p$-adic logarithm, and the Iwasawa invariants.
Finally, we give thorough investigations of torus knots, twist knots, and elliptic curves; 
we complete the list of the cases with $p$-adic limits being in $\Z$ and find the cases such that the base $p$-class numbers are small and $\nu$'s are arbitrarily large. 
\end{abstract}


\maketitle 



\setcounter{tocdepth}{3}
{\small 
\tableofcontents }

\section{Introduction}

Since the dawn of algebraic number theory, the quest for number fields with trivial class numbers has been of great interest. 
In the early 19th century, C.~\!F.~\!Gauss studied quadratic number fields and asked whether there exist infinitely many real quadratic fields with trivial class numbers, which is still an open question. 
Also in the 19th century, H.~\!Weber studied the class numbers of cyclotomic $\Z/2^n\Z$-extensions of $\Q$ (cf.~\cite{HWeber1886}). 
So-called \emph{Weber's class number problem} asks whether the class numbers of cyclotomic $\Z/p^n\Z$-extensions are trivial for arbitrary prime number $p$. 
After many years of efforts, this problem is expected to be affirmative, 
nevertheless the assertion is verified only when $(p,n)$ is $(2,\leq$$6),$ $(3,\leq$$3),$ $(5,\leq$$2),$ and when $7\leq p\leq 19$, $n=1$ at this moment (cf.~\!\cite{FukudaKomatsuMorisawa2014, Gras2021JNT}). 
 
The second author recently approached Weber's problem using continued fraction expansions and pointed out that the sequence of the class numbers in the cyclotomic $\Z_2$-extension of $\Q$ converges in the ring of $2$-adic integers $\mathbb{Z}_2$ \cite[Theorem 5.3]{Yoshizaki2023JTNB}. 
W.~\!Sinnott also announced in 1985 a similar result for a cyclotomic $\Zp$-extension of a CM field and for the ``minus'' class numbers, and Sang-G.~\!Han established an explicit formula \cite[Theorem 4]{Han1991} by using an analytic argument. 
Their results are specific cases of H.~\!Kisilevsky's theorem over any global field, that is, a finite extension of $\Q$ or $\F_{p'}(x)$, $p'$ being a prime number: 

\setcounter{section}{2} 
\begin{thmx}[\Cref{thm.field}, {\cite[Corollary 2]{Kisilevsky1997PJM}}]
Let $k_{p^\infty}$ be a $\Zp$-extension of a global field $k$. 
Then, the sizes of the class groups ${\rm C}(k_{p^n})$, 
those of the non-$p$-subgroups ${\rm C}(k_{p^n})_{\text{non-}p}$, 
and those of the $l$-torsion subgroups ${\rm C}(k_{p^n})_{(l)}$ for each prime number $l$ 
converge in $\Zp$. 
\end{thmx} 

The growth of $p$-torsions has been extensively studied in the context of Iwasawa theory. This theorem defines a numerical invariant, say, \emph{the $p$-adic class number} $\lim_{n\to \infty}|{\rm C}(k_{p^n})_{\text{non-}p}|$ of a $\Zp$-extension with any $p$-torsion growth. 

It is said that Gauss's proof of the quadratic reciprocity law using Gauss sums is based on his insight into the analogy between knots and prime numbers. 
In addition, the analogy between the Alexander--Fox theory for $\Z$-covers and the Iwasawa theory for $\Zp$-extensions has played an important role since the 1960s (cf.~\cite{Mazur1963, Morishita2012}). 
A $p$-adic refinement of Alexander--Fox's theory is of its self-interests, as well as applies to the study of profinite rigidity (cf.~\cite{Ueki5, Ueki6, YiLiu2023Invent}). 
In this view, we establish the Iwasawa-type formula and an analogue of \Cref{thm.field} for unbranched $\Zp$-cover of 3-manifolds. 

\setcounter{section}{3}\setcounter{thm}{0}
\begin{thmx}
Let $(X_{p^n}\to X)_n$ be a $\Zp$-cover of a compact connected 3-manifold $X$. 
Then, 

\noindent 
(1) (\Cref{prop.3mfd.Iwasawa}). There exist some $\lambda,\mu\in \Z_{\geq 0}$ and $\nu \in \Z$ such that for every $n\gg 0$, the size of the $p$-torsion subgroup satisfies  
\[|H_1(X_{p^n})_{(p)}|=p^{\lambda n +\mu p^m +\nu}.\]
(2) (\Cref{thm.3mfd}). The sizes of the torsion subgroups $H_1(X_{p^n})_{\rm tor}$, 
those of the non-$p$ torsion subgroups $H_1(X_{p^n})_{\text{non-}p}$,
and those of the $l$-torsion subgroups $H_1(X_{p^n})_{(l)}$ for each prime number $l$,   
of the 1st homology groups converge in $\Zp$. 
\end{thmx} 
By S.~\!Kionke's theorem \cite[Theorem 1.1 (ii)]{Kionke2020JLMS} and the Poincar\'e duality, 
the $p$-adic limit value $\lim_{n\to \infty} |H_1(X_{p^n})_{\text{non-}p}|$ coincides with Kionke's \emph{$p$-adic torsion}. 
In Sections 2 and 3, we stick to the homological argument and give proofs to these theorems in a parallel manner. 
Afterward, in Section 4, we state a general proposition and discuss alternative proofs. 

In several contexts, the size of the $n$-th layer is calculated by the $n$-th cyclic resultant ${\rm Res}(t^n-1,f(t))=\prod_{\zeta^n=1}f(\zeta)$ 
of a certain polynomial $0\neq f(t) \in \Z[t]$. 
In order to pursue numerical studies, 
we establish the following theorems on the $p$-adic limits of cyclic resultants, 
which are detailed versions of {\cite[Proposition 2]{Kisilevsky1997PJM}}. 
In the proof, we invoke an elementary $p$-adic number theory and the class field theory with modulus. 
Let $\C_p$ denote the $p$-adic completion of an algebraic closure of the $p$-adic numbers $\Q_p$ and fix an embedding $\ol{\Q}\inj \C_p$.

\setcounter{section}{4}\setcounter{thm}{2}
\begin{thmx}[\Cref{thm.res.conv}]
Let $0\neq f(t)\in \Z[t]$. Then, the $p$-power-th cyclic resultants 
${\rm Res}(t^{p^n}-1,f(t))$ converge in $\Zp$. 
The limit value is zero if and only if $p\mid f(1)$. 
In any case, if ${\rm Res}(t^{p^n}-1,f(t))$ $\neq 0$ for any $n$, then  
the non-$p$-parts of ${\rm Res}(t^{p^n}-1,f(t))$ 
converge to a non-zero value in $\Zp$. 
For each prime number $l$, similar assertions for the $l$-parts of ${\rm Res}(t^{p^n}-1,f(t))$ hold. 
\end{thmx} 

\begin{thmx}[\Cref{thm.res}, a short version] 
Suppose $p\nmid f(t)$. 
Write $f(t)=a_0\prod_i (t-\alpha_i)$ in $\ol{\Q}[t]$ and note that $|a_0\prod_{|\alpha_i|_p>1} \alpha_i|_p=1$. 
Let $\xi$ and $\zeta_i$ denote the unique roots of unity of order prime to $p$ 
satisfying $|a_0\prod_{|\alpha_j|_p>1} \alpha_j -\xi|_p<1$ and $|\alpha_i-\zeta_i|_p<1$ for each $i$ with $|\alpha_i|_p=1$. 
Then 
\[\lim_{n\to \infty} {\rm Res}(t^{p^n}-1,f(t)) =
(-1)^{p\, {\rm deg}f +\#\{i\, \mid\, |\alpha_i|_p<1\}
} \xi \prod_{i;\,|\alpha_i|_p=1} (\zeta_i-1)\]
holds in $\Zp$. 
In addition, the non-$p$ part of ${\rm Res}(t^{p^n}-1,f(t))$ converges to 
\[(-1)^{p\, {\rm deg}f +\#\{i\, \mid\, |\alpha_i|_p<1\}
} \xi\, \bigl( \prod_{\substack{i;\, |\alpha_i|_p=1,\\ \ |\alpha_i-1|_p=1}} (\zeta_i-1)\bigr) \, p^{-\nu} \prod_{\substack{i;\, |\alpha_i|_p=1,\\ \  |\alpha_i-1|_p<1}} \log \alpha_i\]
in $\Zp$, where $\log$ denotes the $p$-adic logarithm and 
$\nu$ $\in \Z\cup\{\infty\}$ is Iwasawa's invariant defined by 
$p^{-\nu}=\prod_{i;|\alpha_i-1|_p<1} |\log \alpha_i|_p$. 
If all $\alpha_i$'s with $|\alpha_i-1|_p<1$ are sufficiently close to 1, 
then $p^\nu=|f(1)|_p^{-1}$ holds. 
\end{thmx} 

In the cases of $\Zp$-covers of knots, Fox--Weber's formula asserts that the cyclic resultants of the Alexander polynomials coincide with the sizes of torsion subgroups of the 1st homology groups. 
We calculate the $p$-adic limits of $|H_1(X_{p^n})_{\rm tor}|$ for the $\Zp$-covers of torus knots $T_{a,b}$ and twist knots $J(2,2m)$ to establish Propositions \ref{prop.torusknots}, \ref{prop.twistknots}, \ref{prop.Atiyah}, 
\emph{completing the table of the cases with 
the $p$-adic limits being in $\Z$.}  
Moreover, we give a systematic study of the Iwasawa $\nu$-invariants and answer the following question (Propositions \ref{prop.nu>0}, \ref{prop.nu.knot}):
\emph{
Find $\Zp$-covers $(X_{ep^n}\to X_e)_n$ with $e\in \Z_{>0}$ 
of twist knots $J(2,2n)$ such that the base $p$-class numbers $|H_1(X_e)_{(p)}|$ are small and $\nu$'s are arbitrary large.} 
In Subsection \ref{ss.rem}, 
we discuss several possible analogues of Weber's problem for knots; 
we remark Livingston's results in \cite{Livingston2002GT} and 
point out further problems in view of the Sato--Tate conjecture.  

In the cases of constant $\Zp$-extensions of function fields, the cyclic resultants of the Frobenius polynomials coincide with the sizes of the degree zero divisor class groups. In Section \ref{s.ac}, we recollect basic facts of function fields, 
state an analogue of Fox--Weber's formula for constant extensions of function fields (\Cref{FoxWeberFF}), and study elliptic curves over finite fields. 
We point out conditions for the $p$-adic limit value 
being 0 and 1 using the notions of supersingular primes and anomalous primes, 
as well as \emph{complete the list of the cases with the $p$-adic limits being in $\Z$} (\Cref{prop.EC}, \ref{prop.Atiyah.elliptic}). 
We also give a systematic study of the Iwasawa $\nu$-invariant and answer the following question (Propositions \ref{prop.nu>0.elliptic}, \ref{prop.classify.elliptic}): 
\emph{
Find constant $\Zp$-extensions $(k_{ep^n}/k_e)_n$ with $e\in \Z_{>0}$ of the function field of elliptic curves over $\F_l$ such that the base $p$-class numbers $|{\rm Cl}^0(k_e)_{(p)}|$ are small and $\nu$'s are arbitrarily large.} 

Note that we have intentionally kept our materials to the very basic, such as torus knots, twist knots, and elliptic curves, to raise questions in a broad scope. 
Our results in this article were initially announced by the authors at several conferences in 2021--2022. 
This article contains a detailed revisiting of Kisilevsky's short article \cite{Kisilevsky1997PJM}.  
Our numerical study of the $p$-adic limits gives explicit examples of 
Kionke's $p$-adic torsions introduced in \cite{Kionke2020JLMS}. 
Recent related works are due to A.~\!Gothandaraman \cite{GAsvin2023ANT} and M.~\!Ozaki \cite{Ozaki-padiclimit} (See Remarks \ref{rem.Asvin} and \ref{rem.Ozaki}). 
In addition, C.~Deninger points out that there would exist a common generalization of our work and his \cite{Deninger2020GGD}. 

\subsection*{Acknowledgments} 
The authors would like to express their sincere gratitude to 
Cristopher Deninger, 
Yoshinosuke Hirakawa, Teruhisa Kadokami, Tomokazu Kashio, 
Hershy Kisilevsky, 
Satoshi Kumabe, 
Moemi Matsumoto, 
Daichi Matsuzuki, 
Tomoki Mihara, 
Yasushi Mizusawa, 
Manabu Ozaki, Makoto Sakuma, 
Shin-ichiro Seki, 
Jordan Schettler, 
and the anonymous referee of the journal  
for useful information and fruitful conversations. 
We dedicate this article to Toshie Takata. 
The first and second authors have been partially supported by 
JSPS KAKENHI Grant Number JP19K14538 and 22J10004 respectively. 

\setcounter{section}{1}
\section{Global fields} 

A number field is a finite extension of $\Q$. 
A function field is a finite extension of the rational function field $\F_{p'}(x)$ of one variable over a finite field $\F_{p'}$, $p'$ being a prime number. 
A global field is a number field or a function field. 
For a global field $k$, let ${\rm C}(k)$ denote the ideal class group ${\rm Cl}(k)$ if $k$ is a number field, and 
the degree-zero divisor class group ${\rm Cl}^0(k)$ if $k$ is a function field. 
Note that ${\rm C}(k)$ is always a finite group. We regard ${\rm C}(k)$ as a multiplicative group. 
For any finite abelian group $G$ and a prime number $l$, let $G_{(l)}$ and $G_{\text{non-}p}$ denote the $l$-torsion subgroup and non-$p$ torsion subgroup of $G$ respectively.  
The size of a finite set $X$ is written as $|X|$. 
A $\Zp$-extension $k_{p^\infty}$ of a global field $k$ is 
a direct system $(k_{p^n})_n$ of $\Z/p^n\Z$-extensions or its union $\bigcup_n k_{p^n}$. 
The following theorem was initially proved by Kisilevsky {\cite[Corollary 2]{Kisilevsky1997PJM}}. 
We note that although Kisilevsky's proof is short and clear, we here give 
our original proof with a purpose. 

\begin{thm}
\label{thm.field} 
Let $k_{p^\infty}$ be a $\Zp$-extension of a global field $k$. 
Then, the sizes of the class groups ${\rm C}(k_{p^n})$, 
those of the non-$p$-subgroups ${\rm C}(k_{p^n})_{{\rm non}\text{-}p}$, 
and those of the $l$-torsion subgroups ${\rm C}(k_{p^n})_{(l)}$ for each prime number $l$ converge in $\Zp$. 
\end{thm}

\begin{proof} 
It is well-known (see \Cref{rem.p-part} below) that for any $n\gg 0$, 
the class field theory yields that $|{\rm C}(k_{p^{n-1}})|$ divides $|{\rm C}(k_{p^{n}})|$. 
Hence the sequence $|{\rm C}(k_{p^{n}})_{(p)}|$ is a constant for $n\gg 0$ or it converges to 0 in $\Zp$. 
Thus, it suffices to prove for each prime number $l\neq p$ and $n\in \Z_{>0}$ 
the congruence formula of relative class numbers 
\begin{align}
|{\rm C}(k_{p^{n}})_{(l)}|/|{\rm C}(k_{p^{n-1}})_{(l)}| \equiv 1 \ {\rm mod}\ p^{n}.
\end{align}
Define the relative norm map 
${\rm Nr}:{\rm C}(k_{p^{n}})\to {\rm C}(k_{p^{n-1}}): [\mf{a}]\mapsto \prod_{i=0}^{p-1}\mf{a}^{\tau^i}$, where $\tau$ is a generator of ${\rm Gal}(k_{p^{n}}/k_{p^{n-1}})\cong \Z/p\Z$. 

The map ${\rm Nr}:{\rm C}(k_{p^{n}})_{(l)}\to {\rm C}(k_{p^{n-1}})_{(l)}$ on the $l$-parts is surjective. 
Indeed, there is a natural homomorphism $\iota:{\rm C}(k_{p^{n-1}})_{(l)}\to {\rm C}(k_{p^{n}})_{(l)}$ 
and the composition map ${\rm Nr}\circ \iota: {\rm C}(k_{p^{n-1}})_{(l)}\to {\rm C}(k_{p^{n-1}})_{(l)}$ is given by $x\mapsto x^p$. Since $l\neq p$, this map 
${\rm Nr} \circ \iota $ 
is an isomorphism and hence ${\rm Nr}$ is surjective. 

Note that $|({\rm Ker}\,{\rm Nr} 
)_{(l)}|=|{\rm C}(k_{p^{n}})_{(l)}|/|{\rm C}(k_{p^{n-1}})_{(l)}|$. 
We study the Galois module structure of $({\rm Ker}\,{\rm Nr})_{(l)}$ to obtain the assertion. 
Put $G={\rm Gal}(k_{p^{n}}/k) \cong \Z/p^{n}\Z$ and let $\sigma$ be a generator of $G$. 
For each $[\mf{a}]\in ({\rm Ker}\,{\rm Nr})_{(l)},$ let $G[\mf{a}]$ denote the $G$-orbit of $[\mf{a}]$. 
If $[\mf{a}]\neq 1$, then $|G[\mf{a}]|=p^{n}$. 
Indeed, suppose that $|G[\mf{a}]|<p^{n}$. 
Then $|G[\mf{a}]|$ divides $p^{n-1}$ and we have that $[\mf{a}]=[\mf{a}^{\sigma^{p^{n-1}}}]$. 
Note that $\sigma^{p^{n-1}}$ generates the group $p^{n-1}\Z/p^{n}\Z\cong {\rm Gal}(k_{p^{n}}/k_{p^{n-1}})< G$ and put $\tau=\sigma^{p^{n-1}}$. 
Since $[\mf{a}]\in ({\rm Ker}\,{\rm Nr})_{(l)}$, we have $[\mf{a}]^p=[\prod_{i=0}^{p-1}\mf{a}^{\tau^i}]$ $=1$. 
Since $l\neq p$, we obtain $[\mf{a}]=1$. 

Now the $G$-orbital decomposition yields that ${\rm Ker}\,{\rm Nr}_{(l)}\equiv 1$ mod $p^{n}$, hence the claimed formula $(\ast)$. 
Therefore, both $(|{\rm C}(k_{p^n})_{(l)}|)_n$ and $(|{\rm C}(k_{p^n})_{\text{non-}p}|)_n$ are $p$-adic Cauchy sequences and converge in the completed ring $\Zp$, 
and so does $(|{\rm C}(k_{p^n})|)_n$. 
\end{proof} 

\begin{rem} \label{rem.p-part} 
The following well-known argument completes the first paragraph of the proof. 

(1) For each number field $k$, let $\wt{k}$ denote the Hilbert class field, that is, the maximal unramified abelian extension of $k$. Then the class field theory asserts that ${\rm Cl}(k)\cong {\rm Gal}(\wt{k}/k)$. 
If $k'/k$ is a ramified extension of degree $p$, then we have $\wt{k} \cap k'=k$ and that $\wt{k}k' /k'$ is an unramified extension of degree $|{\rm Cl}(k)|$, and hence $|{\rm Cl}(k)|$ divides $|{\rm Cl}(k')|={\rm deg}( \wt{k'}/k)$. 

If $k_{p^\infty}/k$ is a $\Zp$-extension, then the inertia group of a ramified prime is an open subgroup of $\Zp={\rm Gal}(k_{p^\infty}/k)$, and hence $k_{p^n}/k_{p^{n-1}}$ is a ramified $p$-extension for any $n\gg0$. 

(2) 
For a function field $k$, let $\wt{k}$ denote the maximal unramified abelian extension of $k$. 
Let $\ol{\F}_{p'}$ denote the algebraic closure of $\F_{p'}$, so that we have ${\rm Gal}(\ol{\F}_{p'}/\F_{p'})\cong \wh{\Z}=\varprojlim_{r} \Z/r\Z$. 
Then an analogue of the class field theory assets that ${\rm Cl}^0(k)\cong {\rm Gal}(\wt{k}/k\ol{\F}_{p'})$. 

(i) If $k'/k$ is a constant extension of degree $p$, then by $k'\ol{\F}_{p'}=k\ol{\F}_{p'}$, 
$\wt{k}/k\ol{\F}_{p'}$ is a subextension of $\wt{k'}/k'\ol{\F}_{p'}$, and hence $|{\rm Cl}(k)|$ divides $|{\rm Cl}(k')|$. 

(ii) If $k'/k$ is a geometric ramified extension of degree $p$, then a similar argument to (1) yields that $|{\rm Cl}(k)|$ divides $|{\rm Cl}(k')|$. 

For a $\Zp$-extension $k_{p^\infty}/k$ of a function field, $k_{p^n}/k_{p^{n-1}}$ is a constant extension for all $n\in \Z_{>0}$ and (i) applies, or $k_{p^n}/k_{p^{n-1}}$ is a geometric ramified extension for all $n\gg 0$ and (ii) applies. In the latter case, we always have $p=p'$. 
\end{rem} 

\begin{rem} In a view of the analogy between number fields and function fields, 
Iwasawa pointed out so-called Iwasawa's class number formula (cf.~\cite{Iwasawa1959},\cite[Section 7.2]{Washington}), 
which asserts that if $k_{p^\infty}$ is a $\Zp$-extension of a number field $k$, 
then there exists some $\lambda, \mu, \nu \in \Z_{\geq 0}$ such that for any $n\gg 0$,
\[|{\rm C}(k_{p^n})_{(p)}|=p^{\lambda n + \mu p^n +\nu}\]
holds. 
A similar formula with $\mu=0$ holds for a constant $\Zp$-extension of a function field \cite[Theorem 11.5]{Rosen2002GTM} 
and $\lambda$ is related to the genus of an algebraic curve in several senses. 

In many literature of number theory, the suffix is shifted as $k'_n=k_{p^n}$. 
Note that $\lambda'n + \mu' p^n +\nu'=\lambda (n-1) + \mu p^{n-1} +\nu$ implies 
$\mu=p\mu'$, $\lambda=\lambda'$, $\nu=\nu'+\lambda$. 

Gold--Kisilevsky \cite{GoldKisilevsky1988} pointed out that in a geometric $\Zp$-extension the $p$-parts can grow arbitrarily fast. Even in such a case, \Cref{thm.field} persists. 
\end{rem} 

\begin{rem} Let $l\neq p$ be a prime number. 
Washington \cite{Washington1978Invent} proved that in a cyclotomic $\Zp$-extension of a number field abelian over $k$, for each prime number $l\neq p$, the $l$-part of the class numbers are bounded, and hence the sequence is constant for $n\gg 0$. 
The assertion on the $l$-part in our \Cref{thm.field} is a weak generalization of Washington's one. 
\end{rem} 

\begin{rem} Weber's class number problem for function fields over finite fields is solved;  
Shen--Shi \cite{ShenShi2015JNT} completed the list of the only existing 8 exceptional cases. 
\end{rem}

\section{3-manifolds} 
In this section, we establish a theorem of $p$-adic convergence in the context of 3-dimensional topology. 
A \emph{$\Zp$-cover} of a compact 3-manifold $X$ is a compatible system $(X_{p^n}\to X)_n$ of $\Z/p^n\Z$-covers. 
The following is an analogue of \Cref{thm.field}. 
\begin{thm}[$p$-adic convergence] \label{thm.3mfd}
Let $(X_{p^n}\to X)_n$ be a $\Zp$-cover of a compact connected 3-manifold $X$. 
Then, the sizes of the torsion subgroups $H_1(X_{p^n})_{\rm tor}$, 
those of the non-$p$ torsion subgroups $H_1(X_{p^n})_{\text{non-}p}$,
and those of the $l$-torsion subgroups $H_1(X_{p^n})_{(l)}$ for each prime number $l$,  
of the 1st homology groups converge in $\Zp$. 
\end{thm} 

We aim to prove the assertion in a parallel manner to \Cref{thm.field}. 
In order to discuss the $p$-part, we first establish the Iwasawa type formula for unbranched $\Zp$-covers of 3-manifolds, which is a variant of the results of \cite{HillmanMateiMorishita2006, KadokamiMizusawa2008, Ueki2} in a slightly general setting: 
\begin{thm}[The Iwasawa-type formula] \label{prop.3mfd.Iwasawa} 
Let $(X_{p^n}\to X)_n$ be a $\Zp$-cover of a compact connected 3-manifold. Then, 
there exists some $\mu,\lambda\in \Z_{\geq 0}$ and $\nu\in \Z$ such that 
for every sufficiently large $n$, 
\[H_1(X_{p^n})_{(p)}=p^{\mu p^n+\lambda n + \nu}\] holds. 
\end{thm} 
If $X$ is the exterior of a knot in $S^3$, then for the unique $\Zp$-cover of $X$, we have $|H_1(X_{p^n})_{(p)}|=1$ by \cite[Theorem 7]{Ueki1}.

\begin{proof} Recall the Iwasawa isomorphism $\Lambda:=\Zp[[t^{\Zp}]]\congto \Zp[[T]];$ $t\mapsto 1+T$. 
Define a finitely generated $\Lambda$-module by $\mca{H}=\varprojlim_n H_1(X_{p^{n}},\Zp)$. 
For every $m\in \Z_{\geq 0}$, we have an exact sequence 
\[ 0\to H_1(X_{p^{n+m}},\Zp)/(t^{p^n}-1)H_1(X_{p^{n+m}},\Zp)\to H_1(X_{p^{n}},\Zp) \to \Z/p^{n+m}\Z\to 0.\]
Since the inverse limit is exact for compact topological groups (cf.~{\cite[Chapter 4, Proposition 2.7]{Neukirch}}), 
by taking $\varprojlim$ with respect to $m$, we obtain an exact sequence 
\[ 0\to \mca{H}/(t^{p^n}-1)\mca{H}\to  H_1(X_{p^{n}},\Zp) \to \Zp\to 0,\]
so we have $(\mca{H}/(t^{p^n}-1)\mca{H})_{\rm tor}=H_1(X_{p^{n}},\Zp)_{\rm tor}$. 
By the structure theorem of finitely generated $\Lambda$-modules (cf.~\cite[Theorem 13.12]{Washington}), there is a pseudo isomorphism to a standard form
\[\mca{H}\overset{\sim}{\To} \Lambda^r\oplus \bigoplus_i\Lambda/(p^{n_i}) \oplus \bigoplus_j \Lambda/f_j(1+T)^{m_j},\]
where $n_i, m_j\in \Z_{\geq 0}$, and $f_j\in \Zp[t]$ are so-called distinguish irreducible polynomials.

In the cases of $\Zp$-extensions of number fields and $\Zp$-cover consisting of $\Q$HS$\,^3$'s, we always have that $f_j$ does not vanish on $p$-power-th roots of unity. 
In our case, $f_j(t)$ may vanish at some $p$-power-th root of unity. 
Even in such a case, by the exact sequence 
\[0\to \Lambda/(f_j(t)^{m_j-1}, \frac{t^{p^n}-1}{f_j(t)})\overset{\times f_j(t)}{\to} \Lambda/(f_j(t)^{m_j},t^{p^n}-1) \overset{{\rm mod}\, f_j(t)}{\to} \Lambda/(f_j(t)^{m_j},f_j(t))\to 0,\] 
we see that 
\[|(\Lambda/(f_j(t)^{m_j},t^{p^n}-1))_{\rm tor}|=|\Lambda/(f_j(t)^{m_j-1},\frac{t^{p^n}-1}{f_j(t)})|.\] 
Therefore, the standard argument (cf.~\cite{Washington,Ochiai2023Iwasawa1en}) persists for this case, yielding the Iwasawa-type formula for our groups 
$(\mca{H}/(t^{p^n}-1)\mca{H})_{\rm tor}\cong H_1(X_{p^{n}},\Zp)_{\rm tor}
\cong H_1(X_{p^{n}})_{(p)}$.
\end{proof} 

\begin{proof}[Proof of \Cref{thm.3mfd}] 
By the Iwasawa-type formula \Cref{prop.3mfd.Iwasawa}, $|H_1(X_{p^{n}})_{(p)}|$ is a constant for $n\gg 0$ or converges to zero in $\Zp$. 	

It suffices to show for each prime number $l\neq p$ and $n\in\Z_{>0}$ the congruence formula 
\begin{align}
|H_1(X_{p^{n}})_{(\ell)}|/|H_1(X_{p^{n-1}})_{(\ell)}| \equiv 1 \ {\rm mod}\ p^{n}. 
\end{align} 
Write $h:X_{p^{n}}\to X_{p^{n-1}}$ and $h_*:H_1(X_{p^{n}}) \to H_1(X_{p^{n-1}})$. 
Consider the transfer map $h^!:H_1(X_{p^{n-1}})\to H_1(X_{p^{n}})$ defined by $[c]\mapsto [\sum_{\sigma\in {\rm Gal}(h)} \sigma c_1]$, where $c$ is an open chain and $c_1$ is a lift of $c$. 
Then the composition map is $h_*\circ h^!:H_1(X_{p^{n-1}})\to H_1(X_{p^{n-1}}); [c]\mapsto p[c]$. 
Since $p\neq l$, this map $h_*\circ h^!$ restricts to an automorphism of $H_1(X_{p^{n-1}})_{(l)}$. Hence $h_*:H_1(X_{p^{n}})_{(l)}\to H_1(X_{p^{n-1}})_{(l)}$ is surjective and $h^!:H_1(X_{p^{n-1}})_{(l)}\to H_1(X_{p^{n}})_{(l)}$ is injective. 

Put $G={\rm Gal}(X_{p^{n}}\to X) \cong \Z/p^{n}\Z$ and let $\sigma$ be a generator of $G$. 
For each $[c]\in ({\rm Ker}h_*)_{(l)},$ let $G[c]$ denote the $G$-orbit of $[c]$. 
If $[c]\neq 1$, then $|G[c]|=p^{n}$. 
Indeed, suppose that $|G[c]|<p^{n}$. Then 
$|G[c]|$ divides $p^{n-1}$ and we have that $[c]=[\sigma^{p^{n-1}}c]$. 
Note that $\sigma^{p^{n-1}}$ generates the group $p^{n-1}\Z/p^{n}\Z\cong {\rm Gal}(h)< G$ and put $\tau=\sigma^{p^{n-1}}$. 
Since $[c]\in ({\rm Ker}h_*)_{(l)}$, we have $[c]^p=\prod_{i=0}^{p-1}\tau^i [c]$ $=1$. 
Since $l\neq p$, we obtain $[c]=1$. 

Now the $G$-orbital decomposition yields that ${{\rm Ker}\,h_\ast}_{(l)}\equiv 1$ mod $p^{n}$, hence the claimed formula $(\ast)$. 
Therefore, both $(|H_1(X_{p^n})_{(l)}|)_n$ and $(|H_1(X_{p^n})_{\text{non-}p}|)_n$ are $p$-adic Cauchy sequences and converge in the completed ring $\Zp$, 
and so does $(|H_1(X_{p^n})_{\rm tor}|)_n$. 
\end{proof}

\begin{rem} 
(1) 
Some standard $\Zp$-covers are derived from $\Z$-covers of $X$.    
An example of a $\Zp$-cover that is not derived from a $\Z$-cover may be obtained from a 2-component link $L=K_1\cup K_2$ in $S^3$ with meridians $\mu_1$ and $\mu_2$; 
consider the surjective homomorphism  $\wh{\pi}_1(S^3-L)\surj \Z_5$ defined by $\mu_1\mapsto 1$ and $\mu_2\mapsto \sqrt{-1}$. 

(2) \Cref{thm.3mfd} applies to $\Zp$-cover of the exterior of a finite link in $S^3$. 
In addition, for a link $\mca{L}=\bigcup_{i\in \Z_{> 0}} K_j$ with countably infinite disjoint component in $S^3$, if we define a surjective homomorphism $\tau:{\wh{\pi}_1(X-\mca{L})}\surj \Zp; \mu_i\mapsto p^i$ from the profinite completion, $\mu_i$ being a meridian of $K_i$, then we obtain a branched $\Zp$-cover branched along an infinite link. By considering $\tau_n:{\wh{\pi}_1(X-\bigcup_{i\leq n}K_i)}\surj \Zp; \mu_i\mapsto p^i$ on each layer, 
the same argument in the proof applies, so \Cref{thm.3mfd} extends to such a case. 
\end{rem} 

\section{Alternative proofs} \label{section.alternative}
Here we state a general proposition to discuss alternative proofs of the $p$-adic convergence theorems (Theorems \ref{thm.field}, \ref{thm.3mfd}). 

\begin{prop} \label{prop.conv} 
Let $p$ be a prime number. 
Let $\Gamma$ be a multiplicative group isomorphic to the additive group $\Zp$ of $p$-adic integers.
For each $n\in \Z_{>0}$, put $\Gamma_n=\Gamma^{p^n}$ and $G_n=\Gamma/\Gamma_n$. 

{\rm (1)}\cite[Proposition 1]{Kisilevsky1997PJM} Let $A$ be a discrete $\Gamma$-module such that the $\Gamma_n$-invariant subgroup 
$A_n=A^{\Gamma_n}=\{a\in A\mid \gamma(a)=a\text{\ for\ all\ }\gamma \in \Gamma_n\}$
is a finite group for every $n$. Then we have 
\[ |A_{n}|\equiv |A_{n-1}| \ {\rm mod}\  p^n.\]

{\rm (2)} Let $H$ be a compact $\Gamma$-module such that the $\Gamma_n$-coinvariant quotient group 
$H_n=H_{\Gamma_n}=H/\{(1-g)a\mid g\in \Gamma_n, a\in H\}$ is a finite group for every $n$. Then for any prime number $l\neq p$, the sizes of the $l$-parts satisfy 
\[|{H_{n}}_{(l)}|\equiv |{H_{n-1}}_{(l)}| \ {\rm mod} \ p^n.\] 
\end{prop} 

\begin{proof} 
(1) Let $B=\{a\in A_n \mid \gamma(a)\neq a\text{\ for\ all\ } 1\neq \gamma\in G_n\}$ and write $A_n=B\sqcup C$. 
Since $G_n$ is a cyclic group, every $c\in C$ is fixed by the unique subgroup of $G_n$ of order $p$, 
so we have $C\subset A^{\Gamma_{n-1}}=A_{n-1}$. 
Since $A_{n-1}\cap B=\emptyset$, we have $A_{n-1}\subset C$. Thus we have $C=A_{n-1}$. 
Since $B$ is the disjoint union of orbits of size $p^n$, 
we have $|A_n|=|B|+|A_{n-1}|\equiv |A_{n-1}|$ mod $p^n$. 

(2)
[Proof 1] We omit ``$_{(l)}$''. 
It suffices to show that $|H_n|/|H_{n-1}|=|{\rm Ker}(H_n\surj H_{n-1})|\equiv 1$ mod $p^n$.
Let us prove that if $0 \neq [a] \in {\rm Ker}(H_n\surj H_{n-1})$, then $G_n=\langle t \rangle =\Gamma/\Gamma_n \cong \Z/p^n\Z$ acts on $[a]$ freely. 
Let $[a]\in {\rm Ker}(H_n\surj H_{n-1})$ and suppose $|G_n[a]|<p^n$, so that 
$[a]=t^{p^{n-1}}[a]$. Put $\sigma=t^{p^{n-1}}$. 
Since $H_{n-1}=H/(1-\sigma)H$, we have $a\in (1-\sigma)H$ and hence $\sum_{i=0}^{p-1}\sigma^i a =(1-\sigma^p)H$. Thus we have $p[a]=\sum_{i=0}^{p-1}\sigma^i [a]=0$ in $H_n$. 
Since $l\neq p$, we have $[a]=0$. 

[Proof 2] We continue to omit ``$_{(l)}$''. 
Let $\varphi_n:H_n\surj H_{n-1}; [a]\mapsto [a]$ denote the natural surjection. 
Note that there is a natural map 
$\iota_n:H_{n-1}\to H_n; [a]\mapsto \sum_{i=0}^{p-1}\sigma^i[a]$. 
Since $\varphi_n\circ \iota_n([a])=p[a]$ and $l\neq p$, 
$\varphi_n\circ \iota_n$ is an isomorphism and $\iota_n$ is an injection. 
Thus we have an injective system $(H_n;\iota_n)_n$.  
If we put $A=\varinjlim H_n$, then the assertion (1) applies. 
\red{Indeed, we obviously have $\iota_n(H_n) \subset H_{n+1}^{\,\Gamma_n}$. 
In addition, if $a\in H_{n+1}^{\Gamma_n}$, then by $\iota_n(\varphi_{n+1}(a))=(1+t^{p^n}+\cdots+t^{(p-1)p^n})a=pa$, $\varphi_{n+1}(a)=0$ implies $a=0$. 
Hence, the restriction of $\varphi_{n+1}$ to $H_{n+1}^{\Gamma_n}$ is injective, and we obtain $\iota_n(H_n)\supset H_{n+1}^{\,\Gamma_n}$. 
Thus, we have $\iota_n(H_n) = H_{n+1}^{\,\Gamma_n}$, and hence $H_n=A^{\Gamma_n}$.}   
\end{proof} 

The common argument in our proofs in the previous sections may be generalized to the first proof of (2). 
Kisilevsky \cite{Kisilevsky1997PJM} applied the assertion (1) to the direct limit of the class groups in a $\Zp$-extension to obtain his result. 
In the topology side, we may also consider the direct limit $A=\varinjlim H_1(X_n)_{\rm tor}$ via the transfer maps $h^!:H_1(X_{n-1})\to H_1(X_{n}): [c]\mapsto [\sum_{\sigma \in {\rm Gal}(h)} \sigma c_1]$ and apply (1) to obtain the result. 
Kionke's general framework \cite{Kionke2020JLMS} for the $p$-adic limits of topological invariants 
instead considers the injective system of the cohomology groups $H^i(X_n;\Z)$. 
The second author's proof \cite[Theorem 5.3]{Yoshizaki2023JTNB} for the cyclotomic $\Z_2$-extension of $\Q$ is very different from those in the above and uses the unit groups. We wonder if it extends to general cases and may be translated into analogous contexts. 
In fact, an analogue of the unit group is still in mystery (cf.~\cite{UekiYasuda2023LDTNT}). 

\begin{rem} \label{rem.Ozaki} 
After Sinnott's announcement in 1972, Han \cite[Theorem 4]{Han1991} established an explicit formula for the $p$-adic limit of class numbers in a $\Zp$-extension of a CM field by using an analytic argument. 

Recently, Ozaki \cite{Ozaki-padiclimit} generalized the $p$-adic convergence theorem of class numbers 
to a general extension of a number field with a finitely generated pro-$p$ Galois group by developing an analytic method, 
to reveal relationships amongst several arithmetic invariants; the class numbers, the ratios of $p$-adic regulators, the square roots of discriminants, and the order of algebraic $K_2$-groups of the ring of integers. 
Studying their analogues in the knot theory side would give a new cliff to extend the dictionary of arithmetic topology.
\end{rem} 

\begin{rem} \label{rem.} J.~\!Schettler proved that in a $\Zp$-extension of a $\Zp$-field, Iwasawa's $\lambda$ converges in $\Zp$ \cite[Corollary 11]{Schettler2014IJNT}. It would be interesting to establish analogous results for 3-manifolds or function fields and give numerical investigations. 
\end{rem}

\section{Cyclic resultants} 
Let $p$ be a prime number as before. 
In various situations, in a $\Zp$-tower, the class number or its analogue of the layer of degree $p^n$ is given by the $p^n$-th cyclic resultant of a certain polynomial invariant. 
In this section, we study the $p$-adic limits of $p$-power-th cyclic resultants of a polynomial in $\Z[t]$. 

\subsection{Signatures}
For each $n\in \Z_{>0}$, \emph{the $n$-th cyclic resultant of} $0\neq f(t)\in \Z[t]$ is defined by the determinant of Sylvester matrix, or equivalently, by 
\[{\rm Res}(t^n-1,f(t))=\prod_{\zeta^n=1}f(\zeta),\]
where $\zeta$ runs through $n$-th roots of unity in a fixed algebraic closure $\ol{\Q}$ of $\Q$. 
If $f(t)=a_0\prod_i (t-\alpha_i)$, then ${\rm Res}(t^n-1,f(t))=$ $(-1)^{n\,{\rm deg}f(t)}$ $a_0^n\prod(\alpha_i^n-1)$ holds. 
The $n$-th cyclotomic polynomial $\Phi_n(t)\in \Z[t]$ for each $n\in \Z_{>0}$ is an irreducible polynomial determined by $t^n-1=\prod_{m|n} \Phi_m(t)$ $(m,n\in \Z_{>0})$ recursively. 
The non-$p$ part of an integer $x$ is defined by $x_{non-p}=x|x|_p$. 
For each prime number $l$, the $l$-part of an integer $x$ is defined to be the maximal $l$-power dividing $x$, that is, $|x|_l^{-1}$. 
The following lemma reduces the calculation of the limits of $p$-power-th cyclic resultants to that of the absolute values. 

\begin{lem} \label{lem.sgn} Let $0\neq f(t)\in \Z[t]$. 
{\rm (1)} If ${\rm Res}(t^n-1,f(t))\neq 0$, then we have ${\rm Res}(t^n-1,f(t))>0$ if and only if {\rm (i)} $2\mid n$ and $f(1)f(-1)>0$ or {\rm (ii)} $2\nmid n$ and $f(1)>0$. 

{\rm (2)} Suppose that $n\in \Z_{>0}$ and ${\rm Res}(t^{p^n}-1, f(t)) \neq 0$.  
If $p\neq 2$, then we have ${\rm Res}(t^{p^n}-1,f(t))>0$ if and only if $f(1)>0$. 
If $p=2$, then we have ${\rm Res}(t^{p^n}-1,f(t))>0$ if and only if $f(1)f(-1)>0$. 
\end{lem} 

\begin{proof} 
For any $m\in \Z_{>2}$, we have 
$\prod_{\zeta;\, \Phi_m(\zeta)=0}f(\zeta)={\rm Nr}_{\Q(\zeta_m)/\Q}f(\zeta_m)>0$,
where $\zeta_m$ is an arbitrary taken primitive $m$-th root of unity. 
By ${\rm Res}(t^n-1,f(t))=\prod_{m\mid n}\prod_{\zeta;\, \Phi_m(\zeta)=0}f(\zeta)$, 
we obtain the assertion. 
\end{proof} 

\subsection{$p$-adic convergence}
The $p$-adic convergence theorem (\Cref{thm.res.conv}) for a polynomial may be proved by applying \Cref{prop.conv} to the compact module $H=\varprojlim \Z[t]/(f(t),t^{p^n}-1)$. 
Here, we give another proof by invoking the global field theory with modulus. 
For a number field $k$, let $I(k)$ and $P(k)$ denote the ideal group and the principal ideal group of $k$ respectively. 
In addition, for a divisor $\mf{M}=\prod_i \mf{p}_i^{e_i} \prod_j \infty_j$ of $k$, where $\mf{p}_i$'s are distinct prime ideals of $k$ with $e_i\in \Z_{>0}$ and $\infty_j$'s are distinct real places, 
set $I_\mf{M}(k)=\{\mf{a}\in I(k)\mid (\prod_i \mf{p}_i^{e_i},\mf{a})=1\}$ and 
$P_\mf{M}(k)=\{\mf{a}\in I_\mf{M}(k)\cap P(k)\mid \mf{a}\equiv 1\ {\rm mod}^*\ \mf{M}\}$, 
where $\mf{a}\equiv 1\ {\rm mod}^*\ \mf{M}$ means that there exists some $\alpha\in k$ such that (i) $\mf{a}=(\alpha)$ and the multiplicative $\mf{p}_i$-adic valuation satisfies $v_{\mf{p}_i}(\alpha-1)\geq e_i$ for all $\p_i$ and (ii) $\alpha>0$ at all $\infty_j$. 
Then we have the following. 

\begin{lem}[Artin reciprocity law, cf.~{\cite[Appendix \S 3, Theorem 1(i)]{Washington}}] \label{lem.GCFT} 
Let the notation be as above. Let $k'/k$ be a finite abelian extension and suppose that the conductor $\mf{f}$ of $k'/k$ divides $\mf{M}$. 
Then there is a natural isomorphism called Artin's reciprocity map
\[I_\mf{M}(k)/P_\mf{M}(k){\rm Nr}_{k'/k}(I_\mf{M}(k'))\congto {\rm Gal}(k'/k).\]
\end{lem} 

The following theorem on $p$-adic convergence yields alternative proofs of Theorems \ref{thm.field} and \ref{thm.3mfd} for several situations, as we will exhibit later. 

\begin{thm} 
\label{thm.res.conv} 
Let $0\neq f(t)\in \Z[t]$. 
Then, the $p$-power-th cyclic resultants ${\rm Res}(t^{p^n}-1,f(t))$ converge in $\Zp$. 
The limit values are zero if and only if $p\mid f(1)$. 
In any case, if ${\rm Res}(t^{p^n}-1,f(t)) \neq 0$ for any $n$, then 
the non-$p$-parts of ${\rm Res}(t^{p^n}-1,f(t))$ converge to a non-zero value in $\Zp$. 
For each prime number $l$, 
similar assertions for the $l$-parts of ${\rm Res}(t^{p^n}-1,f(t))$ hold. 
\end{thm} 

\begin{proof} 
If ${\rm Res}(t^{p^n}-1,f(t))=0$ for some $n$, so that the limit value is zero, then we have $\Phi_{p^m}(t)\mid f(t)$ for some $m\mid n$ and hence $p\mid f(1)$. 

Assume ${\rm Res}(t^{p^n}-1,f(t))\neq 0$. 
For each $n\in \Z_{>0}$, let $\zeta_{p^n}$ be an arbitrary taken primitive $p^n$-th root of unity. Then we have 
\[\dfrac{{\rm Res}(t^{p^{n}}-1,f(t))}{{\rm Res}(t^{p^{n-1}}-1,f(t))}=\prod_{0\leq i<p^n; (i,p)=1} f(\zeta_{p^n}^i)={\rm Nr}_{\Q(\zeta_{p^n})/\Q} f(\zeta_{p^n}).\]

If $p\mid f(1)$, then we have $f(t)\equiv (1-t)g(t)$ mod $p$ and $f(t)=(1-t)g(t)+ph(t)$ for some $g(t),h(t)\in \Z[t]$. 
Since $(1-\zeta_{p^n})$ is a unique prime ideal of $\Z[\zeta_{p^n}]$ dividing $(p)$, 
we have $p\mid {\rm Nr}_{\Q(\zeta_{p^n})/\Q} f(\zeta_{p^n})$, 
and hence ${\rm Res}(t^{p^n}-1,f(t))$ converges to zero in $\Zp$. 

Suppose instead that $p\nmid f(1)$. Let us prove that ${\rm Nr}_{\Q(\zeta_{p^n})/\Q} f(\zeta_{p^n}) \equiv 1$ mod $p^n$ for each $n\in \Z_{>1}$. 
Note that we have $(1-\zeta_{p^n})\nmid f(\zeta_{p^n})$ in $\Z[\zeta_{p^n}]$. 
Indeed, if $(1-\zeta_{p^n})\mid f(\zeta_{p^n})$, then 
$f(t)\equiv (1-t)g(t)$ mod $\Phi_{p^n}(t)$ and hence   
$f(1)\equiv 0$ mod $\Phi_{p^n}(1)=p$. 

By $(1-\zeta_{p^n})\nmid f(\zeta_{p^n})$ in $\Z[\zeta_{p^n}]$, we have $(f(\zeta_{p^n})) \in I_{p^n}(\Q(\zeta_{p^n}))$. 
Applying \Cref{lem.GCFT} for $K/k=\Q(\zeta_{p^n})/\Q$ and $\mf{M}=(p^n)\infty$, we obtain a natural isomorphism 
\[I_{p^n}(\Q)/P_{p^n\infty}(\Q){\rm Nr}_{\Q(\zeta_{p^n})/\Q}(I_{p^n}(\Q(\zeta_{p^n})))\congto {\rm Gal}(\Q(\zeta_{p^n})/\Q)\cong (\Z/p^n\Z)^\ast\]
sending ${\rm Nr}_{\Q(\zeta_{p^n})/\Q}f(\zeta_{p^n})$ to $1$ mod $p^n$. 
This means that ${\rm Nr}_{\Q(\zeta_{p^n})/\Q}f(\zeta_{p^n})\equiv 1$ mod $p^n$. 
Hence $({\rm Res}(t^{p^n}-1,f(t)))_n$ is a $p$-adic Cauchy sequence and converges in the $p$-adic completion $\Zp$ of $\Z$. In this case, the limit value is not zero. 

Even if $p|f(1)$, if we replace $f(\zeta_{p^n})$ in above by its non-$p$ part $a_{p^n}=f(\zeta_{p^n})\,(1-\zeta_{p^n})^{-v_n} \in I_{p^n}\Q(\zeta_{p^n})$, where $v_n=v_{(1-\zeta_{p^n})}(f(\zeta_{p^n}))$, 
then a similar argument shows that 
${\rm Nr}_{\Q(\zeta_{p^n})/\Q}a_{p^n}\equiv 1$ mod $p^n$ and hence the assertion on the non-$p$-parts of ${\rm Res}(t^{p^n}-1,f(t))$'s.  

For each $l\neq p$, if $p\nmid f(1)$, then since 
$|{\rm Nr}_{\Q(\zeta_{p^n})/\Q} f(\zeta_{p^n})|_l^{-1} 
={\rm Nr}_{\Q(\zeta_{p^n})/\Q} (f(\zeta_{p^n}))_l  
\equiv 1$ mod $p^n$ holds. 
Here, $(f(\zeta_{p^n}))_l$ 
denotes the $l$-part of the ideal $(f(\zeta_{p^n}))$ of $\Q(\zeta_{p^n})$, 
\red{that is, $(f(\zeta_{p^n}))_l$ is the product of prime divisors of $(f(\zeta_{p^n}))$ over $l$,}  
and ${\rm Nr}_{\Q(\zeta_{p^n})/\Q}$ also denotes the \red{absolute} norm map for ideals.   
A similar argument proves the assertion for the $l$-parts. 

The assertion for the $p$-parts is clear. 
\end{proof} 

\begin{rem} \label{rem.Asvin} 
In a study of $l$-adic convergence in Iwasawa towers of varieties over finite fields, 
G.~\!Asvin recently proved a result close to our heart by using a method different from ours; 
His result \cite[Corollary 5]{GAsvin2023ANT} asserts that if $f(t)$ and $g(t)$ are monic in $\Z_l[t]$, 
then ${\rm Res}(f(t), g(t^{l^{n+1}}))\equiv {\rm Res}(f(t), g(t^{l^n}))$ mod $l^{n+1}$. 
He derives the assertion from a variant of Fermat's little theorem due to Arnold--Zarelua \cite[Theorem 4]{Zarelua2008TMIS}; 
For $A\in {\rm M}_r(\Z_l)$, ${\rm tr}\,A^{l^{n+1}}\equiv {\rm tr}\,A^{l^n}$ mod $l^{n+1}$ holds. 
\end{rem} 

\subsection{Explicit formula}
Let $\C_p$ denote the $p$-adic completion of an algebraic closure of the $p$-adic numbers $\Q_p$ and fix an embedding $\ol{\Q}\inj \C_p$. Let $\ol{\Z}_p$ denote the closure of $p$-adic integers $\Zp$ in $\C_p$. 
Since extensions of $\F_p$ are cyclotomic extensions of degrees prime to $p$, 
an elementary $p$-adic number theory yields the following basic fact. 
\begin{lem}[{cf.~\cite[Lemma 2.10]{Ueki4}}] \label{lem.p-prime-th} 
If $\alpha\in \Cp$ satisfies $|\alpha|_p=1$, then there exists a unique root of unity $\zeta$ of order prime to $p$ 
satisfying $|\alpha-\zeta|_p<1$. 
\end{lem} 

The following lemma is also elementary and classically known. 
\begin{lem} \label{lem.converge} 
Let $\alpha,\zeta \in \C_p$ 
with $|\alpha|_p=|\zeta|_p=1$. 

{\rm (1)} 
If $|\alpha-\zeta|_p<1$, then 
$\lim_{n\to \infty} (\alpha^{p^n} -\zeta^{p^n} )=0$ in $\Cp$. 

{\rm (2)} If $|\alpha-1|_p<1$, then 
$\lim_{n\to \infty} \frac{\alpha^{p^n} -1}{p^n}=\log \alpha$ in $\Cp$, 
where $\log$ denotes the $p$-adic logarithm defined by $\log (1+x)=\sum_{n=1}^\infty \frac{-(-x)^n}{n}$ on $\ol{\Z}_p$. 

{\rm (3)} If $|\alpha-1|_p<p^{-1/(p-1)}$, then $|\log \alpha|_p = |1-\alpha|_p$. 
\end{lem} 
\begin{proof} 
(1) 
Define a $(\omega_n)_n$ by $\omega_1={\rm gcd}\{p(\alpha-\zeta), (\alpha-\zeta)^p\}$ 
and $\omega_{n+1}={\rm gcd}\{p\omega_n, \omega_n^p\}$ for all $n$, 
where gcd of a finite subset $A\subset \ol{\Z}_p$ means the maximal power of a fixed uniformizer dividing all elements of $A$. 
Then we have $\zeta^{p^n}=(\alpha-(\alpha-\zeta))^{p^n}=(\alpha^{p} + p g_1(\alpha,\zeta) +(\alpha-\zeta)^p)^{p^{n-1}}
=(\alpha^p + \omega_1 h_1(\alpha,\zeta))^{p^{n-1}}
=\cdots 
=\alpha^{p^n} + \omega_n h_n(\alpha,\zeta)$ 
for some $g_i(\alpha, \zeta), h_i(\alpha,\zeta)\in \Z[\alpha,\zeta]$. 
Since $\lim_{n\to \infty} \omega_n=0$ and $|h_n(\alpha,\zeta)|_p\leq 1$, we have $\lim_{n\to \infty} |\alpha^{p^n} -\zeta^{p^n}|_p =0$. 

(2) If we put $\varepsilon=\alpha-1$, then 
by an elementary $p$-adic calculus assures that 
\[\lim_{n\to \infty}\frac{\alpha^{p^n}-1}{p^n}=\lim_{n\to \infty} \frac{(1+\varepsilon)^{p^n}-(1+\varepsilon)^0}{p^n}=
\frac{d}{dx}\exp((\log(1+\varepsilon))x)|_{x=0}
=\log (1+\varepsilon)=\log \alpha.\] 

(3) Put $\varepsilon=\alpha-1$. 
Then the strong triangle inequality yields 
\[|\log (1+\varepsilon)|_p=|\sum_{n=1}^\infty (-1)^{n-1}\varepsilon^n/n|_p \leq {\rm sup}\{|\varepsilon^{p^k}/p^k|_p \mid k\in \Z_{\geq 0}\}= |\varepsilon|_p,\] 
and the equality holds if the sequence $|\varepsilon^{p^k}/p^k|_p$ takes distinct values when $k$ moves. 
By the assumption, we have $|\varepsilon|_p>|\varepsilon^p/p|_p$, hence the equality. 
\end{proof}

The following explicit formula is a key to our numerical study. 

\begin{thm} \label{thm.res} 
Let $0\neq f(t)\in \Z[t]$ and let $p^\mu$ denote the maximal $p$-power dividing $f(t)$. 
Write $f(t)=a_0\prod_i (t-\alpha_i)$ in $\ol{\Q}[t]$ and note that $|p^{-\mu} a_0\prod_{|\alpha_i|_p>1} \alpha_i|_p=1$. 
Let $\xi$ and $\zeta_i$ denote the unique roots of unity of orders prime to $p$ 
satisfying $|p^{-\mu}a_0\prod_{|\alpha_j|_p>1} \alpha_j -\xi|_p<1$ and $|\alpha_i-\zeta_i|_p<1$ for each $i$ with $|\alpha_i|_p=1$. 

{\rm (1)} {\rm (i)} If $p\mid f(t)$, so that $\mu>0$, then $\lim_{n\to \infty} {\rm Res}(t^{p^n}-1,f(t)) =0$ holds in $\Zp$.  

{\rm (ii)} If $p\nmid f(t)$, so that $\mu=0$, then 
\[
\lim_{n\to \infty} {\rm Res}(t^{p^n}-1,f(t)) =
(-1)^{p\, {\rm deg}f + 
\#\{i\, \mid\, |\alpha_i|_p<1\}
} \xi \prod_{i;\,|\alpha_i|_p=1} (\zeta_i-1)
\] 
holds in $\Zp$, and the limit value is zero if and only if $\zeta_i=1$ for some $i$. 

{\rm (2)} In any case, the non-$p$ part 
${\rm Res}(t^{p^n}-1,f(t))_{\text{non-}p}= {\rm Res}(t^{p^n}-1,f(t))\, | {\rm Res}(t^{p^n}-1,f(t))|_p$ converges to 
\[
(-1)^{p\, {\rm deg}f +\#\{i\, \mid\, |\alpha_i|_p<1\}} 
\xi\, \bigl( \prod_{\substack{i;\, |\alpha_i|_p=1,\\ \ |\alpha_i-1|_p=1}} (\zeta_i-1)\bigr) 
\, p^{-\nu} \prod_{\substack{i;\, |\alpha_i|_p=1,\\ \  |\alpha_i-1|_p<1}} \log \alpha_i 
\] 
in $\Zp$, where $\log$ denotes the $p$-adic logarithm and 
$\nu \in \Z$ $\cup \{\infty\}$ is defined by 
$p^{-\nu}=\prod_{\substack{i;\, |\alpha_i-1|<1}} |\log \alpha_i|_p$. 
If all $\alpha_i$'s with $|\alpha_i-1|_p<1$ are sufficiently close to 1, that is, 
if they all satisfy $|\alpha_i-1|_p<p^{-1/(p-1)}$, then $p^\nu=|f(1)|_p^{-1}$ holds. 

Put $\lambda=\#\{i\mid |\alpha_i-1|_p<1\}$. 
Then these $\lambda, \mu, \nu$ are the Iwasawa invariants of $f(t)$, 
that is, $|{\rm Res}(t^{p^n}-1, f(t))|_p^{-1}=p^{\lambda n+\mu p^n +\nu}$ holds for any $n\gg 0$. 
\end{thm} 

\begin{rem}
If $f(t)$ is monic and ${\rm deg}f$ is even, then our theorem 
recovers \cite[Proposition 2]{Kisilevsky1997PJM}: 
\[
{\rm Res}(t^{p^n}-1,f(t))_{\text{non-}p}=
(-1)^{\lambda} \bigl( \prod_{\substack{i;\, |\alpha_i-1|_p=1}} (1-\zeta_i)\bigr) \, p^{-\nu} \prod_{\substack{i;\, |\alpha_i-1|_p<1}} \log \alpha_i. 
\] 
\end{rem}

\begin{proof} It suffice to verify the case with $\mu=0$. 
By \Cref{lem.p-prime-th}, 
if $\alpha\in \Cp$ satisfies $|\alpha|_p=1$, then there exists a unique root of unity $\zeta$ of orders prime to $p$ 
satisfying $|\alpha-\zeta|_p<1$, and hence 
\Cref{lem.converge} (1) 
yields $\lim_{n\to \infty} \alpha^{p^n}-\zeta^{p^n}=0$ in $\Cp$. 
Note that 
\[{\rm Res}(t^{p^n}-1,f(t))=(-1)^{p^n{\rm deg}f}a_0^{p^n}\prod_i(\alpha_i^{p^n}-1)=(-1)^{p\,{\rm deg}f}a_0^{p^n}\prod_i(\alpha_i^{p^n}-1)\]
 for $n>0$. 
Since $p\nmid f(t)$, the Newton polygon verifies 
$|a_0\prod_{|\alpha_i|_p>1} \alpha_i|_p=1$. 
Hence we have 
\[
\begin{split}
\lim_{n\to \infty}a_0^{p^n}\prod_i(\alpha_i^{p^n}-1)
&=\lim_{n\to \infty}a_0^{p^n}\prod_{i;\, |\alpha_i|_p>1}(\alpha_i^{p^n}-1) \prod_{i;\, |\alpha_i|_p=1}(\alpha_i^{p^n}-1) \prod_{i;\, |\alpha_i|_p<1}(\alpha_i^{p^n}-1)\\ 
&=\lim_{n\to \infty} \xi^{p^n} \prod_{i;\, |\alpha_i|_p=1} (\zeta_i^{p^n}-1) \prod_{i;\, |\alpha_i|_p<1}(-1). 
\end{split}
\] 
Take 
$m\in \Z$ with $p\nmid m$ and $\xi^m=\zeta_i^m=1$ for all $i$, 
and note that $p^n\equiv 1$ mod $m$ holds if $n\equiv 0$ mod $\varphi(m)$. 
Since the sequence $(\xi^{p^n} \prod_i (\zeta_i^{p^n}-1))_n$ is periodic and converges by 
Theorem 
\ref{thm.res.conv}, 
we have $\xi^{p^n} \prod_i (\zeta_i^{p^n}-1) = \xi^{p^{\varphi(m)}} \prod_i (\zeta_i^{p^{\varphi(m)}}-1) =\xi \prod_i (\zeta_i-1)$ for any $n \in \Z_{\geq 0}$. 
Therefore, the limit value is $(-1)^{p\,{\rm deg}f+\#\{i\,\mid\, |\alpha_i|_p<1\}} \xi\prod_{i;\, |\alpha_i|=1}(\zeta_i-1)$. 

For each root $\alpha_i$ with $|\alpha_i-1|_p<1$, 
by \Cref{lem.converge} (3), we have 
$\lim_{n\to \infty} (\alpha_i^{p^n}-1)/p^n=\log \alpha_i$. 
If we put $p^{\nu}=|\prod_{\substack{i;\, |\alpha_i-1|<1}} \log \alpha_i|_p^{-1}$, then we obtain the limit value as asserted.  
In addition, if all $\alpha_i$'s with $|\alpha_i-1|_p<1$ are sufficiently close to 1, then by \Cref{lem.converge} (3), we have $p^\nu=|\prod_{\substack{i;\, |\alpha_i-1|<1}}  (\alpha_i-1)|=|f(1)|_p^{-1}$. 

The $p$-adic Weierstrass preparation theorem \cite[Theorem 7.3]{Washington} and a standard argument of Iwasawa theory show that there exists some $\lambda, \mu, \nu \in \Z$ satisfying the equality $f(1+T)$ $\dot{=}$ $p^\mu (T^\lambda+p(\text{lower\ terms}))$ up to multiplication by units in $\Zp[[T]]$ and 
$|{\rm Res}(t^{p^n}-1,f(t))|_p^{-1}=p^{\lambda n + \mu p^n + \nu}$ for any $n\gg 0$. These $\lambda, \mu, \nu$ clearly coincide with those in above. 
\end{proof}

\begin{rem} In the case of a $\Zp$-extension or a $\Zp$-cover, in general, 
Iwasawa's $\nu$ is the sum of several contributions: that of the torsion of the base space, that of the pseudo isomorphism between the Iwasawa/Alexander module and the standard module, and that given above. 
We will study examples with large $\nu$'s in Subsubsections \ref{sss.knot.nu} and \ref{sss.ec.nu}. 
\end{rem} 

\begin{rem} The Mahler measure of a polynomial $f(t)$ is defined by the integral along the unit circle as $m(f(t))=\int_{|z|=1}\log|f(z)|\frac{dz}{z}$ and coincides with the limit of the average of the values of $\log|f(z)|$ at roots of unity. 
Its $p$-adic analogue due to Besser--Deninger is given by Shnirel'man's integral, that is, the $p$-adic limit of the average of values at roots of unity of orders prime to $p$ (cf.~\cite{Ueki4}). 
Our $p$-adic limits in \Cref{thm.res} may be seen as $p$-adic analogues of the Mahler measures in another direction.
\end{rem} 

\begin{cor} \label{cor.res} 
Let $0\neq f(t)\in \Z[t]$. 

{\rm (1)} Suppose that $f(t)$ has leading coefficient 1. 
If $f(t)\equiv \Phi_m(t)$ mod $p$ for $m\in \Z_{>0}$ with $p\nmid m$, then 
\[\lim_{n\to \infty} {\rm Res}(t^{p^n}-1,f(t))=\Phi_m(1)=
\begin{cases}
l \text{\ if\ } m=l^e \text{\ for\ a\ prime\ number\ } l \text{\ and\ } e\in \Z_{>0}\\
1 \text{\ if\  otherwise}
\end{cases}\!\!\!\! \text{in}\ \Zp. \]

{\rm (2)} If $f(t)\equiv \xi\prod_i(t-\zeta_i)$ modulo a divisor of $(p)$, where $\xi$ and $\zeta_i$'s are roots of unity of order prime to $p$, then $\lim_{n\to \infty} {\rm Res}(t^{p^n}-1,f(t))=\xi\prod_i(1-\zeta_i)$ in $\Zp$ holds.  
\end{cor} 

\begin{proof} (1) By Hensel's lemma, $f(t)=\prod_i (t-\alpha_i)\equiv \Phi_m(t)$ mod $p$ implies that there is a bijection $\Xi$ between the sets of the roots of $f(t)$ and $\Phi_m(t)$ such that $|\alpha_i-\Xi(\alpha_i)|_p<1$ holds. 
By \Cref{thm.res} (1)(ii), we obtain   
$\lim_{n\to \infty} {\rm Res}(t^{p^n}-1,f(t))
= (-1)^{p\, \varphi(m)} \prod_{\zeta;\, \Phi_m(\zeta)=0} (\zeta-1)
=(-1)^{(p+1) \varphi(m)}\Phi_m(1)$. 
Indeed, 
if $m=2$, then by $p\nmid m$, $p$ is odd. If $m\neq 2$, then $\varphi(m)$ is even. 
In both cases, we have $(-1)^{(p+1) \varphi(m)}=1$, hence the assertion. 

The assertion (2) may be treated similarly.  
\end{proof} 

The following lemma is useful to study $\Zp$-covers in a $\wh{\Z}=\varprojlim_{n\in \Z} \Z/n\Z$\,-cover. 
\begin{lem} \label{lem.mpn}
Let $m\in \Z_{>0}$ with $p\nmid m$ and $\alpha\in \ol{\Q}$. Then 
\[{\rm Res}(t^{mp^n}-1,t-\alpha)
=\prod_{\zeta^m=1,\, \xi^{p^r}=1} (\zeta \xi -\alpha)
=\prod_{\xi^{p^r}=1}(\xi-\alpha^m)
={\rm Res}(t^{p^n}-1,t-\alpha^m).\] 
\end{lem} 

\section{Knots} 
In this section, we apply our theorems to $\Zp$-covers of knots to examine concrete examples and point out remarks on analogues of Weber's class number problem. 

\subsection{Alexander polynomial and Fox--Weber's formula}
Let $K$ be a knot in $S^3$ and let $M_n\to S^3$ denote the branched $\Z/n\Z$-cover, 
that is, the Fox completion of the $\Z/n\Z$-cover $X_n\to X=S^3-K$. 
Let $\Delta_K(t)$ denote the Alexander polynomial of $K$ normalized by $\Delta_K(1)=1$.
If $\Delta_K(t)$ does not vanish on $n$-th roots of unity, then we have 
\begin{prop}[Fox--Weber's formula, cf.~\cite{Weber1979}] \label{prop.FoxWeber}
\[|H_1(M_n)|=|H_1(X_n)_{\rm tor}|=|{\rm Res}(t^n-1,\Delta_K(t))|.\]
\end{prop}
Since $\Delta_K(1)=1$, \Cref{lem.sgn} assures that 
we have ${\rm Res}(t^n-1,\Delta_K(t))<0$ if and only if $2\mid n$ and $\Delta_K(-1)<0$. 
Thus, both our Theorems \ref{thm.3mfd} and \ref{thm.res} apply. 
We will exhibit concrete examples in the succeeding subsections. 

We remark that Fox--Weber's formula has several variants (cf.~\!Sakuma \cite{Sakuma1979, Sakuma1981}, Mayberry--Murasugi \cite{MayberryMurasugi1982}, and Porti \cite{Porti2004} for links and graphs; 
Tange and the first author \cite{TangeRyoto2018JKTR, Ueki10} for representations of knot groups). 
We may replace $t^{p^n}-1$ by $(t^{p^n}-1)/{\rm gcd}(t^{p^n}-1,f(t))$ in Theorems \ref{thm.res.conv} and \ref{thm.res} and apply to these situations. 

\subsection{Torus knots} 
Let $(a,b)$ be a coprime pair of integers. 
The Alexander polynomial 
\[\Delta_{K}=\dfrac{(1-t)(1-t^{ab})}{(1-t^a)(1-t^b)}=\prod_{\substack{m\mid ab\\ m\nmid a,\, m\nmid b}}\Phi_m(t)\]
of the $(a,b)$-torus knot $K=T_{a,b}$ is the product of cyclotomic polynomials. 
\begin{figure}[h] 
\includegraphics[width=70mm
]{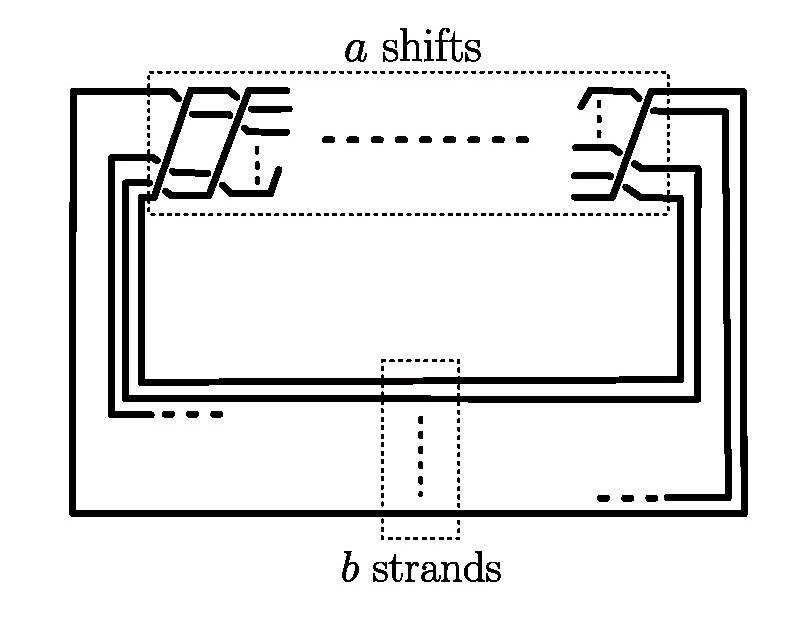} 
\caption*{Torus knot $T_{a,b}$} 
\end{figure} 
For each $n\in \Z_{>0}$, let $\varphi(n)$ denote Euler's totient function. 
We invoke Apostol's result; 
\begin{lem}[{\cite[Theorem 4]{Apostol1970pams}}] \label{Lem.Apostol} 
Suppose that $m>n>1$ and $(m,n)>1$. Then,  
${\rm Res}(\Phi_m,\Phi_n)=p^{\varphi(n)}$ if $m/n$ is a power of a prime $p$, and 
${\rm Res}(\Phi_m,\Phi_n)=1$ if otherwise. 
\end{lem} 

\begin{prop} \label{prop.torusknots}
Let $p$ be a prime number and let $(a,b)$ be a coprime pair of positive integers.
Assume that $p\nmid b$ and write $a=p^r a'$ with $r,a'\in \Z$, $p\nmid a'$. 
Let $(X_{p^n}\to X)_n$ denote the $\Zp$-cover of the exterior of the torus knot $T_{a,b}$ in $S^3$. 
Then 
$|H_1(X_{p^n})_{\rm tor}|=b^{p^{{\rm min}\{n,r\}}-1}$ holds for every $n\in\Z_{\geq0}$, and hence 
\[\lim_{n\to \infty}|H_1(X_{p^n})_{\rm tor}|=b^{p^r-1} \text{\ \ holds\ in\ }\Zp.\] 
In particular, for each pair $(a,b)$, we have $\lim_{n\to \infty}|H_1(X_{p^n})_{\rm tor}|=1$ for almost all $p$'s. 
\end{prop}

\begin{proof} Note that we have $\Delta_K(-1)=(1+t^a+\cdots t^{a(b-1)})/(1+t+\cdots+t^{b-1})|_{t=-1}=1$ or $b>0$ according as 
$2\nmid a$ or $2\mid a$. 
By Fox's formula and \Cref{lem.sgn}, we have 
\begin{equation*}
\begin{split}
|H_1(X_{p^n})_{\rm tor}|&={\rm Res}(t^{p^n}-1, \Delta_K(t))\\ 
&={\rm Res}(\prod_{0\leq i\leq n}\Phi_{p^i}(t), \prod_{\substack{m\mid ab\\ m\nmid a,\, m\nmid b}}\Phi_m(t))\\ 
&=\prod_{0\leq i\leq n} \prod_{\substack{m\mid ab\\ m\nmid a,\, m\nmid b}} {\rm Res}(\Phi_{p^i}(t), \Phi_m(t))).
\end{split}
\end{equation*}
Since $(a,b)$ is a coprime pair, \Cref{Lem.Apostol} assures that ${\rm Res}(\Phi_{p^i}(t), \Phi_m(t)))\neq 1$ if and only if $m=p^i l^j$ for some prime number $l$ and an integer $j\in  \Z_{>0}$ satisfying $l^j\mid b$. 
Let $v_l(b)$ denote the standard multiplicative $l$-adic valuation of $b$. 
Then, by $\sum_{0\leq i\leq n}\varphi(p^i)=p^n-1$, we have 

\begin{equation*}
\begin{split}
|H_1(X_{p^n})_{\rm tor}|
&=\prod_{0\leq i\leq {\rm min}\{n,r\}} \prod_{l\mid b} \prod_{0\leq j \leq v_l(b)} {\rm Res}(\Phi_{p^i}(t), \Phi_{p^i l^j}(t)))\\
&=\prod_{0\leq i\leq {\rm min}\{n,r\}} \prod_{l\mid b} \prod_{0 \leq j \leq v_l(b)} l^{\varphi(p^i)}\\
&=b^{p^{{\rm min}\{n,r\}}-1}. 
\end{split}
\end{equation*}
Hence we obtain the assertion. 
\end{proof}
 
\begin{eg}
Let $K=T_{2,3}=J(2,2)=3_1$ (trefoil). Then we have $\Delta_K(t)=t^2-t+1=\Phi_6(t)$.
We have ${\rm Res}(t^{2^n}-1,\Delta_K(t))=3=3^{2-1}$, ${\rm Res}(t^{3^n}-1,\Delta_K(t))=4=2^{3-1}$, and ${\rm Res}(t^{p^n}-1,\Delta_K(t))=1=3^{p^0-1}$ for $p\neq 2,3$ for all $n\in \Z_{>0}$. 
\end{eg}

\subsection{Twist knots} For each $m\in \Z$, the Alexander polynomial of the twist knot $K=J(2,2m)$ is given by 
$\Delta_{J(2,2m)}(t)=mt^2+(1-2m)t+m$.
The convention is due to \cite{HosteShanahan2004JKTR}, so that we have $J(2,2)=3_1$ and $J(2,-2)=4_1$ for instance. 
\begin{figure}[h]
\includegraphics[width=50mm
]{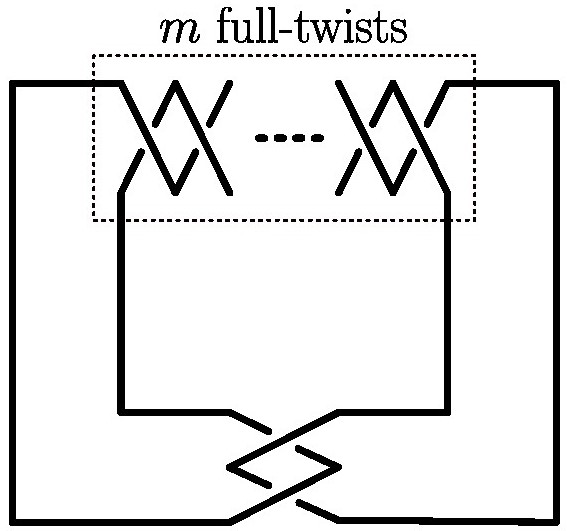} 
\caption*{Twist knot $J(2,2m)$}
\end{figure} 
\subsubsection{Observations for $K=4_1$} 
We first examine $K=4_1$ to demonstrate the usage of our results and raise questions. 

\begin{eg} Let $K=J(2,-2)=4_1$ (the figure-eight knot). Then we have $\Delta_K(t)=-t^2+3t-1$ and 
\begin{center}
\begin{tabular}{c|ccccccc}
$p$&2&3&5&7&$\cdots$ \ \\ \hline 
$\lim_{n\to \infty}|H_1(X_{p^n})_{\rm tor}|$&$-3$&$-2$&$-4$&$\sqrt{2}-2$&$\cdots$ \ 
\end{tabular}, 
\end{center} 
where $\alpha=\sqrt{2}\in \Z_7$ denotes the element satisfying $\alpha^2 =2$ and $\alpha\equiv 3$ mod 7. 
By using PARI/GP \cite{PARI2.15}, we may verify that 
\begin{center} 
\begin{tabular}{c|ccccccc}
$n$&1&2&3&4&5&6&$\cdots$\\
\hline 
${\rm Res}(t^{7^n}-1,\Delta_K(t))$ mod $7^n$&1&8&106&2164&4565&38179& $\cdots$ 
\end{tabular}.
\end{center} 
We have $\lim_{n\to \infty}|H_1(X_{p^n})_{\rm tor}| \in \Z$ only for $p=2,3,5$. 
\end{eg} 

\begin{proof} 
Since $\Delta_K(-1)=-5<0$, we have 
${\rm lim}_{n\to \infty}|H_1(X_{2^n})_{\rm tor}|=-{\rm lim}_{n\to \infty}{\rm Res}(t^{2^n}-1, \Delta_K(t))$ in $\Z_2$ and ${\rm lim}_{n\to \infty}|H_1(X_{p^n})_{\rm tor}|={\rm lim}_{n\to \infty}{\rm Res}(t^{p^n}-1, \Delta_K(t))$ in $\Z_p$ for $p\neq 2$ by \Cref{lem.sgn}. 


In what follows, we utilize \Cref{cor.res} (2) several times.   
If $p=2$, then $\Delta_K(t)\equiv t^2+t+1 =\Phi_3(t)$ mod 2, so 
${\rm lim}_{n\to \infty}|H_1(X_{2^n})_{\rm tor}|
=-{\rm lim}_{n\to \infty}{\rm Res}(t^{2^n}-1, \Delta_K(t))
=-\Phi_3(1)=-3$
in $\Z_2$. 

If $p=3$, then $\Delta_K(t)\equiv -(t^2+1)=-\Phi_4(t)$ mod 3, so 
${\rm lim}_{n\to \infty}|H_1(X_{3^n})_{\rm tor}|
={\rm lim}_{n\to \infty}{\rm Res}(t^{3^n}-1, \Delta_K(t))
=-\Phi_4(1)  
=-2$ in $\Z_3$. 
%
If $p=5$, then $\Delta_K(t)\equiv -(t+1)^2 = -\Phi_2(t)^2$ mod 5, so 
${\rm lim}_{n\to \infty}|H_1(X_{5^n})_{\rm tor}|
={\rm lim}_{n\to \infty}{\rm Res}(t^{5^n}-1, \Delta_K(t))
=-\Phi_2(1)^2 =-4$ in $\Z_5$. 

If $p=7$, then $\Delta_K(t)\equiv -((t+2)^2-3)\equiv -\phi_8^+$ mod 7, where $\Phi_8=t^4+1=\phi_8^+ \phi_8^-$, $\phi_8^\pm=t^2\pm \sqrt{2}t+1$. 
Let $\zeta$ be a primitive 8th root of unity satisfying $\zeta+\ol{\zeta}\equiv -4\equiv 3$ mod 7. 
Then ${\rm lim}_{n\to \infty}|H_1(X_{7^n})_{\rm tor}|=-(\zeta-1)(\ol{\zeta}-1)=-(2-(\zeta+\ol{\zeta}))=-(2-(\sqrt{2}))=-2+\sqrt{2}$. 

For a general $p$, 
let $\zeta\in \ol{\Q}$ be a root of unity with order prime to $p$ satisfying 
$\Delta_K(t)\equiv -(t-\zeta)(t-\zeta^{-1})$ mod $p$. 
Then we have $\lim_{n\to \infty}|H_1(X_{p^n})_{\rm tor}|=\pm(1-\zeta)(1-\zeta^{-1})$. 
Elementary argument shows that there exists a primitive $m$-th root of unity $\zeta\in \overline{\Q}$ such that $(1-\zeta)(1-\zeta^{-1})\in \Z$ if and only if $m=1,2,3,4,6$. 
If $p\geq 7$, then for any $m=1,2,3,4,6$-th root of unity $\zeta$, we have $\Delta_K(\zeta)
=-1,$ $5,$ $2\pm 2\sqrt{-3},$ $\pm\sqrt{-1},$ $\pm3\sqrt{-1},$ $-1\pm\sqrt{-3}$
$\not\equiv 0$ mod $p$, so $\pm(1-\zeta)(1-\zeta^{-1}) \not \in \Z$.   
Hence we have $\lim_{n\to \infty}|H_1(X_{p^n})_{\rm tor}| \in \Z$ if and only if $p=2,3,5$. 
\end{proof}

\begin{eg} Let us examine the $\Z/3^1 2^n \Z$-covers of $K=4_1$. 
Write $\Delta_K(t)=-t^2+3t-1=-(t-\alpha)(t-\beta)$. 
Put $\Delta_3(t)=-(t-\alpha^3)(t-\beta^3)$. 
Then we have $\Delta_3(t)=-t^2+(\alpha^3+\beta^3)t-\alpha^3\beta^3=-t^2+18t-1\equiv -(t-1)^2$ mod 2, 
$\Delta_3(1)=16>0$, $\Delta_3(-1)=-20<0$. 
By Lemmas \ref{lem.sgn} and \ref{lem.mpn}, we have 
$|H_1(X_{3\cdot2^n})_{\rm tor}|=-{\rm Res}(t^{3\cdot 2^n}-1,\Delta_K(t))
=-{\rm Res}(t^{2^n}-1,\Delta_3(t))$ for $n>0$. 
By $|\alpha^3+\beta^3|_2=$ $|18|_2=1/2$ and $\alpha^3\beta^3=1$, we have $|\alpha^3|_2=|\beta^3|_2=1$. 
Since $\alpha^3-1$ and $\beta^3-1$ are roots of 
$-\Delta_3(t)$ $=t^2-16t-16$, we see that 
$|\alpha^3-1|_2=|\beta^3-1|_2=1/4$ $<2^{-1/(2-1)}=1/2$ and hence $2^{-\nu}=|\Delta_3(1)|_2=2^{-4}$. 
Thus, by \Cref{thm.res} (2), we have that    
$|H_1(X_{3\cdot2^n})_{\text{non-}2}|
=|H_1(X_{3\cdot2^n})_{{\rm tor}}|\,2^{-(2n+4)}
=-{\rm Res}(t^{3\cdot 2^n}-1,$ $\Delta_K(t)) 2^{-(2n+4)}
=-{\rm Res}(t^{2^n}-1,$ $\Delta_3(t))2^{-(2n+4)}$,
and this value converges to 
$-\frac{(\log \alpha^3)(\log \beta^3)}{2^4}=\frac{-9}{16}\log\alpha\log\beta$, 
where $\log$ denotes the 2-adic logarithm extended to $\C_2$ so that $\log 2=0$. 
\begin{center} 
\begin{tabular}{c|cccccccccccc} 
$n$&0&1&2&3&4&5&6&7&8&9&10&$\cdots$\\ \hline 
$-{\rm Res}(t^{3\cdot 2^n}-1,\Delta_K(t)) 2^{-(2n+4)}$
&1&5&405&{\scriptsize 10498005}& 
$\cdots$&&&&&&&$\cdots$\\
mod $2^n$&1&1&1&5&5&21&21&85&213&213&213&$\cdots$\\
\end{tabular} 
\end{center} 
\end{eg} 

Kionke \cite{Kionke2020JLMS} discusses whether the $p$-adic Betti number belongs to $\Z$, 
as the $p$-adic analogue of Atiyah's conjecture. 
It also would be interesting to ask when the $p$-adic torsion belongs to $\Z$. 
The observations for $J(2,-2)=4_1$ raise the following questions, to which we will give answers in the rest of this subsection. 
\begin{q} \label{q.knots}
Consider the cyclic covers $X_n\to X=S^3-J(2,2m)$. 

(1)  
Find all pairs $(p,m)$ with $\lim_{n\to \infty} |H_1(X_{p^n})_{\rm tor}| \in \Z \subset \Zp$ and their limit values. 
Find all pairs $(p,m)$ with $\lim_{n\to \infty} |H_1(X_{p^n})_{\rm tor}|=1$. 

(2) Find conditions of $(e,p)$ with $p\nmid e$ such that $\lim_{n\to \infty} |H_1(X_{e p^n})_{\rm tor}| =0$ holds, and study the values of $\nu$. Can $\nu$ be arbitrarily large, with $|H_1(X_{e})_{\rm tor}|$ being small? 
\end{q} 

\subsubsection{Cases with $\lim |H_1(X_{p^n})_{\rm tor}| \in \Z$} 
We discuss both of the cases with $p\mid m$ and $p\nmid m$. 
\begin{prop} \label{prop.twistknots}
If $p\mid m$, then the $\Zp$-covers of $K=J(2,2m)$ satisfies 
\[\lim_{n\to \infty}|H_1(X_{p^n})_{\rm tor}|=
\begin{cases}
-{\rm sgn}(4m-1)&\text{if}\ p=2,\\
1&\text{if}\ p\neq 2
\end{cases} \text{\ \ in\ }\Zp.\] 
\end{prop} 

\begin{proof}
Write $\Delta_K(t)=m(t-\alpha)(t-\beta)$. 
Then by $|\alpha+\beta|_p=|(2m-1)/m|_p=|1/m|_p>1$ and $\alpha\beta=1$, 
we may assume that $|\alpha|_p>|\beta|_p$ and hence $|\alpha|_p=|\alpha+\beta|_p=|1/m|_p$. 
Then ${\rm lim}_{n\to \infty} {\rm Res}(t^{p^n}-1,\Delta_K(t))={\rm lim}_{n\to \infty} m^{p^n}(\alpha^{p^n}-1)(\beta^{p^n}-1)
={\rm lim}_{n\to \infty} -(m \alpha)^{p^n}$. 
The minimal polynomial of $m\alpha$ is $(t-m\alpha)(t-m\beta)=t^2-m(\alpha+\beta)+m^2\alpha\beta=t^2+(1-2m)t+m^2\equiv t^2+t=t(t+1)$ mod $p$. 
Thus $m\alpha\equiv -1$ mod $p$ and hence ${\rm lim}_{n\to \infty}-(m \alpha)^{p^n}=-1$ if $p=2$, 
${\rm lim}_{n\to \infty}-(m \alpha)^{p^n}=1$ if $p\neq 2$ in $\Zp$. 
Note that if $p=2$, then we have an additional term ${\rm sgn}\Delta(-1)={\rm sgn}(4m-1)$ by \Cref{lem.sgn}. 
\end{proof}

\begin{eg} Let $K=J(2,4)=5_2$. Then $\Delta_K(t)=2t^2-3t+2$, $\Delta_K(-1)=7>0$, and 
\begin{center}
\begin{tabular}{c|cccccc}
$n$&1&2&3&4&$\cdots$\\
\hline  
${\rm Res}(t^{2^n}-1,\Delta_K(t))$&7&63&63&60543&$\cdots$\\
\hline 
mod $2^n$&1&3&7&15&$\cdots$
\end{tabular}.\\[2mm]
\end{center}

Let $K=J(2,-4)$. Then $\Delta_K(t)=-2t^2+5t-2$, $\Delta_K(-1)=-9<0$ and 

\begin{center} 
\begin{tabular}{c|cccccc}
$n$&1&2&3&4&$\cdots$\\
\hline  
$-{\rm Res}(t^{2^n}-1,\Delta_K(t))$&$-9$&$-225$ &$-65025$&$-4294836225$& $\cdots$\\
\hline 
mod $2^n$&1&1&1&1&$\cdots$
\end{tabular}.\\[2mm]
\end{center}

Let $K=J(2,6)$. Then $\Delta_K(t)=3t^2-5t+3$ and

\begin{center} 
\begin{tabular}{c|cccccc}
$n$&1&2&3&4&$\cdots$\\
\hline  
${\rm Res}(t^{3^n}-1,\Delta_K(t))$&64&{\scriptsize 18496} &{\scriptsize 30417519283264}&{\scriptsize 1729618048727305550814328969659247936576}& $\cdots$\\
\hline 
mod $3^n$&1&1&1&1&$\cdots$
\end{tabular}.
\end{center}
\end{eg}

\begin{prop} \label{prop.Atiyah} 
Suppose $p\nmid m$. Then the $\Zp$-cover of $K=J(2,2m)$ satisfies 
$\lim=\lim_{n\to \infty}|H_1(X_{p^n})_{\rm tor}| \in \Z$ in $\Zp$ 
if and only if one of the following holds. 
\begin{itemize}
\item $p=2$\,{\rm ;}\ \ $\lim={\rm sgn}(4m-1)\cdot 3=\pm 3$. 
\item $p=3$\,{\rm ;}\ \   $3\mid m-1$, $\lim =4$ or $3\mid m+1$, $\lim =-2$. 
\item $p=5$\,{\rm ;}\ \  $5\mid m+1$, $\lim =-4$. 
\item $p\neq 2,3$\,{\rm ;}\ \  $p\mid m-1$, $\lim =1$. 
\end{itemize} 
\end{prop} 

\begin{proof} In what follows, $x \equiv y$ stands for $x \equiv y$ mod $p$ if otherwise mentioned. 
Let $\xi$ and $\zeta$ denote the unique root of unity of order prime to $p$ with $|m-\xi|_p<1$ and 
$\Delta_K(t)\equiv \xi (t-\zeta)(t-\zeta^{-1})$ mod $p$. 
Again by \Cref{cor.res} (2), we have   
$\lim{\rm Res}(t^{p^n}-1,\Delta_K(t))=\xi (1-\zeta)(1-\zeta^{-1})$. 
If this value is in $\Z$, then we have $\xi=\pm 1$ and $\zeta^6=1$, 
so $\Delta_K(t)/m \equiv (t-1)^2$, $(t+1)^2$, $t^2+t+1$, $t^2+1$, or $t^2-t+1$. 
Note that $\xi \equiv \pm1$ is equivalent to that $p\mid m^2-1$. 

Since ${\rm Res}(\Delta_K(t),(t+1)^2)=(4m-1)^2$, 
${\rm Res}(\Delta_K(t),t^2+t+1)=(3m-1)^2$, 
${\rm Res}(\Delta_K(t),t^2+1)=(2m-1)^2$, and
${\rm Res}(\Delta_K(t),t^2-t+1)=(m-1)^2$, 
we have $\Delta_K(t)\equiv m(t+1)^2$ if $p\mid 4m-1$, $\Delta_K(t)\equiv m(t^2+t+1)$ if $p\mid 3m-1$, $\Delta_K(t) \equiv m(t^2+1)$ if $p\mid 2m-1$, and $\Delta_K(t)\equiv m(t^2-t+1)$ if $p\mid m-1$, 
while $\Delta_K(t)\equiv m(t-1)^2$ is not the case. 

Suppose $p\mid m-1$, so that $m\equiv 1$. If $\Delta_K(t)\equiv (t+1)^2$, then by $0\equiv (4m-1)^2\equiv 3^2$, we have $p=3$. 
If $\Delta_K(t)\equiv t^2+t+1$, then by $0\equiv (3m-1)^2\equiv 2^2$, we have $p=2$. 
If  $\Delta_K(t)\equiv t^2+1$, then $0\equiv (2m-1)^2\equiv 1$, which is not the case. 
If $\Delta_K(t)\equiv t^2-t+1$, then by $0\equiv (m-1)^2$, we just have $p\mid m-1$. 
(In this case, we need $p\neq 2, 3$.) 

Suppose instead $p\mid m+1$, so that $m\equiv -1$. 
If $-\Delta_K(t)\equiv (t+1)^2$, then by $0\equiv (4m-1)^2\equiv 5^2$, we have $p=5$. 
If $-\Delta_K(t)\equiv t^2+t+1$, then by $0\equiv (3m-1)^2\equiv 2^4$, we have $p=2$. 
If $-\Delta_K(t)\equiv t^2+1$, then by $0\equiv (2m-1)^2\equiv 3^2$, we have $p=3$. 
If $-\Delta_K(t)\equiv t^2-t+1$, then 
by $0\equiv (m-1)^2\equiv 2^2$, we have $p=2$. 
(In this case, by $2\mid 6$, \Cref{cor.res} does not apply, while $-\Delta_K(t)\equiv t^2+t+1$ also holds.) 

If $p=2$, then $\Delta_K(t)\equiv t^2-t+1$. \Cref{lem.sgn} yields 
$|H_1(X_{2^n})_{\rm tor}|={\rm sgn}(\Delta_K(-1)){\rm Res}(t^{2^n}-1,$ $\Delta_K(t))={\rm sgn}(4m-1){\rm Res}(t^{2^n}-1,\Delta_K(t))$. 
Hence by \Cref{cor.res}, we have 
$\lim={\rm sgn}(4m-1)\cdot 3$. 
Similarly, 
if $p=3$ and $3\mid m-1$, then $\lim= (t+1)^2|_{t=1}=2^2=4$. 
If $p=3$ and $3\mid m+1$, then $\lim= -(t^2+1)|_{t=1}=-2$. 
If $p=5$ and $5\mid m+1$, then $\lim= -(t+1)^2|_{t=1}=-4$. 
If $p\neq 2,3$ and $p\mid m-1$, then $\lim= t^2-t+1|_{t=1}=1$. 
Combining these above, we obtain the assertion and complete the table. 
\end{proof} 

By Propositions \ref{prop.twistknots} and \ref{prop.Atiyah}, we may conclude the following. 

\begin{cor} \label{cor.twistknots} 
We have $\lim_{n\to \infty}|H_1(X_{p^n})_{\rm tor}|=1$ if and only if one of the following holds. 
\begin{itemize}
\item $p=2$, $m$ is even, and $4m-1<0$.
\item $p\neq 2$ and $p|m$. 
\item $p\neq 2,3$ and $p|(m-1)$. 
\end{itemize}
Particularly, for each $m\neq 0$, we have $\lim_{n\to \infty}|H_1(X_{p^n})_{\rm tor}|=1$ for only finitely many $p$'s. 
\end{cor} 

We discuss further problems about the cases with $\lim=1$ in Subsection \ref{ss.rem}. 

\subsubsection{Can $\nu$ be large with $r_e$ being small?} \label{sss.knot.nu} 
We next investigate the Iwasawa $\nu$-invariants of $\Zp$-covers $(X_{ep^n}\to X_e)_n$ of $K=J(2,2m)$ with 
$e\in \Z_{>0}$, whilst essential cases would be those with $p\nmid e$. 
Put $r_n={\rm Res}(t^{n}-1,\Delta_K(t))$. 
We have $\nu>0$ if $\lim_{n\to \infty} |H_1(X_{ep^n})_{\rm tor}|=0$ in $\Zp$. 
More precisely, we have the following. 
Note that $r_{ep^n}=0$ for some $n$ is equivalent to $m=1$ and $6\mid ep$.

\begin{prop} \label{prop.nu>0} 
For the $\Zp$-cover $(X_{ep^n}\to X_e)_n$ of $K=J(2,2m)$ and $r_{ep^n}\neq 0$, 
the following conditions are equivalent. 
\begin{itemize}
\item $\lim |H_1(X_{ep^n})_{\rm tor}|=0$ in $\Zp$. 
\item $|H_1(X_{e})_{\rm tor}|\equiv 0$ mod $p$.
\item $(X_{ep^n}\to X_e)_n$ has $\nu>0$. 
\end{itemize}
Furthermore, except for the following special cases, we have $p^{-\nu}=|H_1(X_{e})_{\rm tor}|_p$.  
\begin{itemize}
\item $p=3$, $2\mid e$, $|r_{2}|_3=|r_e|_3=1/3$. 
\item $p=2$, $|r_e|_2=1/4$.  
\end{itemize} 
\end{prop} 

\begin{proof} 
Write $\Delta_K(t)=m(t-\alpha)(t-\beta)$ and put $\Delta_e(t)=m^e(t-\alpha^e)(t-\beta^e)$, so that 
we have $r_e=m^e(1-\alpha^e)(1-\beta^e)=m^e(2-(\alpha^e+\beta^e))$, 
$\Delta_e(t)=m^e t^2+(r_e-2m^e)t+m^e$, 
and $\Delta_e(t+1)=m^e(t-(\alpha^e-1))(t-(\beta^e-1))=m^et^2+r_e t+r_e$.  

Suppose that $p\mid r_e$. Then we have $|\alpha^e-1|_p=|\beta^e-1|_p=|r_e|_p^{1/2}$. 
If $p>3$, then $|r_e|_p^{1/2}\leq p^{-1/2}< p^{-1/(p-1)}$, and hence $|\log \alpha^e|_p=|\log \beta^e|_p=|r_e|_p^{1/2}$ and $p^{-\nu}=|H_1(X_{e})_{\rm tor}|_p$. 
If instead $p=3$ and $3^2\mid r_e$, then by $|r_e|_p^{1/2}\leq p^{-1}< p^{-1/(p-1)}$, we obtain a similar result. 
If instead $p=2$ and $2^4 \mid r_e$, then by $|r_e|_p^{1/2}\leq 2^{-2}< 2^{-1}=p^{-1/(p-1)}$, we obtain a similar result. 
Thus, we obtain the assertion. 
\end{proof} 

\begin{rem} 
By a $p$-adic analogue of Lindemann's theorem (cf.~\cite{Mahler1933Crelle, Nesterenko2008IzvMath}) 
asserting that $\alpha\in\ol{\Q}_p$ with $0<|\alpha|_p<1$ implies $\log\alpha\not\in \ol{\Q}$, 
one might expect that the conditions in \Cref{prop.nu>0} (resp. \Cref{prop.nu>0.elliptic}) 
are equivalent to 
$\lim |H_1(X_{ep^n})_{\text{non-}p}| \not\in \ol{\Q}$ (resp. $\lim |{\rm Cl}^0(k_{ep^n})_{\text{non-}p}| \not\in \ol{\Q}$). 
\end{rem}

\begin{eg} 
Let $K=J(2,2)=3_1$ with $\Delta_K(t)=t^2-t+1$. For $e\leq 10$, we have 

\begin{center} 
\begin{tabular}{c||c|c|c|c|c|c|c|c|c|c}
$e$&1&2&3&4&5&6&7&8&9&10\\ \hline 
${\rm Res}(t^e-1,\Delta_K(t))$&1&3&$2^2$&$3$&1&0&1&3&$2^2$&3
\end{tabular}, 
\end{center} 
\begin{itemize}
\item $\nu=1$ for $(e,p)=(2,3),(4,3),(6,3),(8,3)$, 
\item $\nu=2$ for $(e,p)=(3,2), (9,2)$. 
\end{itemize} 

Let $K=J(2,-2)=4_1$ with $\Delta_K(t)=-t^2+3t-1$. For $e\leq 10$, we have 

\begin{center} 
\begin{tabular}{c||c|c|c|c|c|c|c|c|c|c}
$e$&1&2&3&4&5&6&7&8&9&10\\ \hline 
${\rm Res}(t^e-1,\Delta_K(t))$&1&$-5$&$2^4$&$-3^2 5$&$11^2$&$-2^8 5$&$29^2$&$-3^2 5^1 7^2$&$2^4 19^2$&$-5^3 11^2$
\end{tabular}, 
\end{center} 
\begin{itemize}
\item $\nu=1$ for $(e,p)=(2,5),(4,5),(6,5),(8,5)$, 
\item $\nu=2$ for $(e,p)=(4,3), (5,11), (7,29), (8,3),(8,7),(9,19), (10,11)$, 
\item $\nu=4$ for $(e,p)=(3,2), (9,2)$. 
\end{itemize} 

Let $K=J(2,4)=5_2$ with $\Delta_K(t)=2t^2-3t+2$. For $e\leq 10$, we have

\begin{center} 
\begin{tabular}{c||c|c|c|c|c|c|c|c|c|c}
$e$&1&2&3&4&5&6&7&8&9&10\\ \hline 
${\rm Res}(t^e-1,\Delta_K(t))$&1&7&$5^2$&$3^2 7$&$11^2$&$5^27$&$13^2$&$3^2 7$&$5^2$&$7^1 11^2$
\end{tabular}, 
\end{center} 
\begin{itemize}
\item $\nu=1$ for $(e,p)=(2,7),(4,7),(6,7),(8,7)$, 
\item $\nu=2$ for $(e,p)=(3,5), (4,3), (5,11), (6,5), (7,13), (8,3), (9,5), (10,11)$. 
\end{itemize} 
\end{eg}

\begin{eg}[Large base $p$-class number] \label{eg.largenu} 
For $K=J(2,2m)$, we have $r_2=4m-1$, $r_3=9m^2-6m+1=(3m-1)^2$, and $r_4=(2m-1)^2(4m-1)$. 
Under the assumption of \Cref{prop.nu>0}, 
if $p\mid 4m-1$ with exponent 1, then we have $\nu=1$ for even $e$. 
If instead $p\nmid 4m-1$ and $p\mid r_e$, then $\nu$ is even. 

For any $p>3$ and $a\in \Z_{\geq 0}$, if we put $m=(p^a+1)/2$, then we have $p^a=2m-1$ and $|r_4|_p=p^{-2a}$. 
Hence for $e=4$, we have $\nu=2a$. 

For any $p$ with $p\equiv 3$ mod 4 and $a\in \Z_{\geq 0}$, if we put $m=(3^{2a+1}+1)/4$, then we have $r_2=4m-1=p^{2a+1}$. Hence for $e=2$, we have $\nu\geq {2a+1}$. 

For $p=2$ and $a\in \Z_{\geq 0}$, if we put $m=(2^{2c+1}+1)/3$, then we have $r_3=(3m-1)^2=2^{2+2b}=2^{2+4c}$. Hence for $e=3$, we have $\nu\geq 2+4c$. 
\end{eg}

\begin{eg}[Small base $p$-class number and large $\nu$]
 \label{eg.except} 
Let us examine the two exceptional cases in \Cref{prop.nu>0}. 
In these cases, $|r_e|_p=1/p$ holds and $\nu$ becomes arbitrarily large. 

(1) Suppose $p=3$ and $|r_2|_3=1/3$. 
Then, since $r_2=4m-1$, we have $m=3a+1$ with $a=3b$ or $3b+1$, and hence $m=9b+1$ or  $9b+4$, $b\in \Z$. 
(i) If $m=9b+1$, then $r_6=3^5b^2(12b+1)(27b+2)^2$, $|r_6|_3=|3^5 b^2|_3$. 
By \Cref{prop.nu>0} and $\lambda=2$, we have $\nu=v_3(r_6)-2=5+2v_3(b)-2=3+2v_3(b)$.
For instance, if we put $b=3^c$ with $c\in \Z$, then we obtain $m=3^{c+2}+1$ and $\nu=3+2c$ for $(p,e)=(3,2)$. 
More concretely, if we put $c=8$, then we see that the $\Z_3$-cover $(X_{2^1 3^n}\to X_3)_n$ of $K=J(2,2(3^{10}+1))=J(2,118100)$ with $e=2$ has $\nu=19$. 
(ii) If instead $m=9b+4$, then by $|r_{6}|_3=1/3^3$, $(X_{2^1 3^n}\to X_3)_n$ has $\nu=1$. 

(2) Suppose $p=2$ and $|r_3|_2=1/4$. 
Then by $r_3=(3m-1)^2$, we have $3m-1=2(2a+1)$, 
$a=3b$, $m=4b+1$, $b\in \Z$. 
Since $r_6=2^6 b^2 (6b+1)^2(16b+3)$ and $\lambda=2$, 
\Cref{prop.nu>0} yields that 
$\nu=v_2(r_6)-2=6+2 v_2(b)-2=4+2 v_2(b)$. 
For instance, if we put $b=2^c$ with $c\in \Z$, 
then we have $m=2^{c+2} +1$ and $\nu=4+2c$. 
More concretely, if we put $c=48$, then the $\Z_2$-cover $(X_{2^n 3^1}\to X_3)_n$ of $K=J(2,2(2^{50}+1))$ with $e=3$ has $\nu=100$. 
\end{eg}

By Examples \ref{eg.largenu} and \ref{eg.except}, we may conclude the following. 
\begin{prop} \label{prop.nu.knot}
For any $p$ and arbitrary large $N>0$, we may find $K=J(2,2m)$ and $p\nmid e$ such that the $\Zp$-cover $(X_{ep^n}\to X_e)_n$ has $\nu>N$. 
Furthermore, for $(p,e)=(2,3),(3,2)$, we may find such $\Zp$-covers with the base $p$-class number $|H_1(X_e)_{(p)}|=4,3$ respectively. 
\end{prop} 

For a general knot $K$ and a small $p$, we may have  
a slightly large $\nu$, namely, $p^{-\nu}<|H_1(X_e)_{\rm tor}|_p$ holds. 
Nevertheless, such $p$ is bounded by the degree of $\Delta_K(t)$, and hence 
a fixed $K$ has bounded $\nu$'s when $(p,e)$ moves. 

\begin{rem} \label{rem.nu-links} 
For any $\Zp$-cover $(X_{ep^n}\to X_e)_n$ of a knot $K$ in $S^3$, we always have $\nu \geq 0$. 
This is not the case for a general link in $S^3$. Indeed, 
let $L_m=K_1\cup K_2$ be the $m$-twisted Whitehead link in $S^3$ with the Alexander polynomial $\Delta(t_1,t_2)=m(1-t_1)(1-t_2)$ and let $X_n\to X=S^3-L_m$ denote the ``total linking number'' $\Z/n\Z$-cover, that is, the cover corresponding to the kernel of the surjective homomorphism $\pi_1(X)\surj \Z/n\Z$ sending all meridians of $L_m$ to 1. 
If we put $m=p^\mu$ with $\mu\in \Z_{\geq 0}$, 
then a formula of Mayberry--Murasugi \cite{MayberryMurasugi1982} or Porti \cite{Porti2004} yields that 
$|H_1(X_{p^n})_{\rm tor}|=p^n\prod_{\zeta;\,\zeta^{p^n}=1,\, \zeta\neq 1}\Delta(\zeta,\zeta)=p^{3n+\mu p^n-\mu}$ and hence $\nu=-\mu$ may take any non-positive integer if $\mu$ moves. 
The cases of links with multivariable Alexander polynomials will be extensively discussed in \cite{TatenoUeki-Zpd}.  
\end{rem}

\subsection{Livingston's results} 
\label{ss.rem}
In the knot theory side, we may choose the extension degree $p$ and the branch locus $K$ independently, 
so we may consider several analogues of Weber's class number problem. 
Livingston's result in the following considers the set of all prime numbers. 
We may also verify this assertion by using \Cref{Lem.Apostol}. 

\begin{prop}[{Livingston \cite[Theorem 1.2]{Livingston2002GT}}] 
Let $K$ be a knot in $S^3$. Then, the equality 
$|H_1(X_{p^n})_{\rm tor}|=1$ holds for every prime number $p$ and positive integer $n$ 
if and only if 
every non-trivial factor of the Alexander polynomial $\Delta_K(t)$ is 
the $m$-th cyclotomic polynomial with $m$ being divisible by at least three distinct prime numbers. 
\end{prop} 

\begin{eg} Note that a knot with any prescribed Alexander polynomial may be constructed by Rolfsen's method in \cite{Rolfsen1976}. Namely, if we have $\Delta(t)\in \Z[t]$ with $\Delta(1)=1$ and $\Delta(t)\,\dot{=}\,\Delta(1/t)$, then we have a knot $K$ with $\Delta_K(t)=\Delta(t)$. 

There are many knots with $\Delta_K(t)=1$. 
A knot $K$ with $\Delta_K(t)=\Phi_{30}(t)=t^8 + t^7 - t^5 - t^4 - t^3 + t + 1$ would be the initial example with $\Delta_K(t)\neq 1$ satisfying $|H_1(X_{p^n})_{\rm tor}|=1$ for all $p$ and $n$. 
A systematic study of such knots would be of further interest. 

If $p$ is a fixed prime number, then a knot $K$ with $\Delta_K(t)=\Phi_{30}(t)+p(t^6-t^5-t^3+t^2)$ satisfies $\lim_{n\to \infty}|H_1(X_{p^n})_{\rm tor}|=1$ and $(|H_1(X_{p^n})|)_n\neq (1)$. 
We wonder if there exists a knot with $\lim_{n\to \infty}|H_1(X_{p^n})_{\rm tor}|=1$ and $(|H_1(X_{p^n})_{\rm tor}|)_n\neq (1)$ for infinitely many $p$'s. 
\end{eg} 

A subtle question from a viewpoint of the Sato--Tate conjecture in number theory (cf.~Subsubsection \ref{sss.EC.liminZ}) 
is to ask \emph{whether there exist a knot $K$ such that the set $P(K)$ of $p$'s with $\lim_{n\to \infty}|H_1(X_{p^n})_{\rm tor}|=1$ is an infinite set but the density of $P(K)$ in the set of all prime numbers is zero.}  
As we have seen in \Cref{prop.torusknots} and \Cref{cor.twistknots}, torus knots $T_{a,b}$ and twists knots $J(2,2m)\neq 0_1$ are two extreme cases on the opposite sides; they satisfy $\lim_{n\to \infty}|H_1(X_{p^n})_{\rm tor}|=1$ and $\neq 1$ for almost all $p$, respectively.  
We wonder if there is a class with an intermediate behavior. 

Another approach to formulating an analogue of the Sato--Tate conjecture is to consider an infinite family $(K_i)_{\in \Z_{>0}}$ of disjoint knots in $S^3$ satisfying some equidistribution theorem, such as the Chebotarev law (cf.~\cite{Mazur2012, McMullen2013CM, Ueki7, Ueki9}). 
It would be interesting to study the density of knots with the $p$-adic torsion being 1 for each fixed $p$. 
A possible clue is Dehornoy's study on the Lorenz knots ($\fallingdotseq$ modular knots) in \cite{Dehornoy2015Fourier}, which is an analogue of the Riemann/Weil conjecture in a view of Noguchi \cite{Noguchi2005}. 

For a branched $\Zp$-cover $(M_{p^n}\to M_1)_n$, to ask whether all $M_{p^n}$'s are $\Q$HS$^3$'s, or equivalently, to whether every $H_1(M_{p^n})$ is a finite group, is also a weak analogue of Weber's problem. 
By $\Delta_K(1)=1$, we have the following. 
\begin{prop}[{\cite[Corollary 3.2]{Livingston2002GT}}]
The branched $\Z/p^n\Z$-cover of a knot in $S^3$ is always a $\Q$HS$\,^3$. 
\end{prop}
This proposition means that Kionke's $p$-adic Betti number is zero. 
It would be also interesting to study $p$-adic refinements of \cite[Theorem 1.1]{Livingston2002GT} and \cite{KimSG2009JKTR} on concordance. 

\section{Algebraic curves} \label{s.ac}
Theorems on the $p$-adic limit of cyclic resultants are applicable to algebraic curves (function fields) as well. 
We examine the cases of elliptic curves as examples and point out conditions for the $p$-adic limit value being zero and one. 

\subsection{A formula for function fields} 
Let us recollect some properties of function fields to obtain an analogue of Fox--Weber's formula. Basic references are due to Rosen \cite{Rosen2002GTM} and 
Stichtenoth \cite{Stichtenoth2009GTM}. 

Let $k$ be a finite extension of $\F_l(x)$, $l$ being a prime number. (We use $l$ instead of $p'$.) 
Let $\F(k)$ denote the constant field of $k$ and put $q=l^e=|\F(k)|$. 
Let  $\mca{D}_k$, $\mca{P}_k$, and $\mca{E}_k$ denote the set of divisors, that of principal divisors, 
and that of effective divisors of $k$ respectively. 
Put $\mca{D}_k^n=\{A\in \mca{D}_k\mid {\rm deg}A=n\}$ and $\mca{E}^n_k=\{A\in \mca{E}_k\mid {\rm deg}A=n\}$ for each $n\in \Z_{\geq 0}$. 
Let $g(k)$ denote the genus of $k$. 
The congruent zeta function of $k$ is defined by 
\[\zeta_k(s)=\sum_{A\in \mca{E}_k} 
\dfrac{1}{(q^{{\rm deg}A})^s}\]
and satisfies 
\[\zeta_k(s)=\sum_{n=0}^\infty \dfrac{|\mca{E}^n_k|}{(q^n)^s}=\prod_{P\in \mca{P}_k}\left(1-\dfrac{1}{(q^{{\rm deg}P})^s}\right)^{-1}.\]
This Dirichlet series is known to absolutely converge to a holomorphic function on ${\rm Re}(s)>1$. Moreover, we have the following. 
\begin{prop}[{Hasse--Weil, cf.~\cite{Weil1948PIMUS}, \cite[Theorem 5.9]{Rosen2002GTM}}] There exists $L_k(t)\in \Z[t]$ of degree $2g(k)$ satisfying  
\[\zeta_k(s)=\dfrac{L_k(q^{-s})}{(1-q^{-s})(1-q^{1-s})}\]
on ${\rm Re}(s)>1$. 
\end{prop} 
This $L_k(t)$ is called \emph{the $L$-polynomial of $k$}. 
The right-hand side is an analytic continuation of $\zeta_k(s)$ to $\C$ as a meromorphic function. 
In addition, we have the following. 
\begin{itemize}
\item $L_k(0)=1$, $L_k(1)=|{\rm Cl}^0(k)|$, \ \ 
\item $L_k(t)=q^{g(k)}t^{2g(k)}L_k(\dfrac{1}{qt})$, 
\item $L_k(t)=\prod_{i=1}^{2g(k)}(1-\alpha_i t)$ for some algebraic integers $\alpha_i$ with $|\alpha_i|=\sqrt{q}$. 
\end{itemize}
If we write $L_k(t)=a_{2g(k)}t^{2g(k)}+\cdots + a_1t +a_0$, then we have $a_0=1$, $a_1=|\mca{E}^1_k|-(q+1)$, 
$a_{2g(k)}=q^{g(k)}$, and $a_{2g(k)-i}=q^{g(k)-i}a_i$ for $0\leq i\leq 2g(k)$. 
If $k_n/k$ be a constant extension of degree $n$, then we have $k_n=k\,\F(k_n)$, $g(k_n)=g(k)$, $L_{k_n}(t)=\prod_{i=1}^{2g(k)}(1-\alpha_i^n t)$, and hence 
\[|{\rm Cl}^0(k_n)|=L_{k_n}(1)=\prod_{i=1}^{2g(k)}(1-\alpha_i^n)={\rm Res}(t^n-1, t^{2g(k)}L_k(1/t)).\]
For any prime number $l'\neq l$, 
\emph{the Frobenius polynomial} $F_k(t)$ of $k$ is defined as 
the characteristic polynomial of the geometric Frobenius action on the $l'$-adic \'{e}tale cohomology of the algebraic curve corresponding to $k$ and satisfies $F_k(t)=t^{2g(k)}L_k(1/t)$ (cf.~\cite{AubryPerret2004FFA}). 
Hence we obtain the following (See also \cite[Section 2]{Kisilevsky1997PJM}). 

\begin{prop} \label{FoxWeberFF} 
Let $k$ be a function field and $k_n/k$ a constant extension of degree $n$. Then 
\[|{\rm Cl}^0(k_n)|=|{\rm Res}(t^n-1,F_k(t))|.\]
\end{prop} 
This formula may be seen as an analogue of Fox--Weber's formula (\Cref{prop.FoxWeber}) for a constant extension of a function field. 

In a geometric extension of a function field, the genera may grow (cf.~\cite{KostersWan2018}). 
In the cases of knots, if we remove the assumption that $\Delta_K(t)$ does not vanish on roots of unity, then the 1st Betti numbers grow. 
We may expect further analogies there. 

\subsection{Elliptic curves} We observe the cases of elliptic curves as examples. 
Basic literatures are Silverman \cite{SilvermanAEC2009GTM} and Diamond--Shurman \cite{DiamondShurman2005GTM}.  
Let $E$ be an elliptic curve over a finite field $\F$ and let $k$ denote the function field of $E$, 
that is, we have $k_E={\rm Frac}(\F[x,y]/(E(x,y)))$. 
In addition, let $E(\F)$ denote the Model--Weil group, that is, the union of the set of $\F$-rational points of $E$ and $\{\infty\}$. Write $F_E(t)=F_k(t)$. 
Then, we have $|{\rm Cl}^0(k_E)|=F_{E}(1)=|E(\F)|$. 
Let $\F'/\F$ be a finite extension and let $E_{\F'}$ denote the same elliptic curve with the coefficient field replaced by $\F'$. 
Then we have $k_{E_{\F'}}=\F'k_{E}$ and 
$|{\rm Cl}^0(k_{E_{\F'}})|=F_{E_{\F'}}(1)=|E(\F')|$.
If $\F=\F_l$, then the Frobenius polynomial of $E$ is given by 
\[F_E(t)=t^2-(l+1-|E(\F_l)|)t+l.\]
Write $F_E(t)=(t-\alpha)(t-\beta)$. 
For each $e\in \Z_{>0}$ and $n\in \Z_{\geq 0}$, 
write  $E_{\F_{l^{ep^n}}}=E_{l^{ep^n}}$. 
Then, 
\[F_{E_{l^{ep^n}}}(t)=(t-\alpha^{e p^n})(t-\beta^{ep^n})\]
coincides with the Frobenius polynomial of the constant $\Z/p^n\Z$-extension $k_{E_{l^{ep^n}}}$ of $k_{E_{l^e}}$. 

In what follows, we first examine a fixed elliptic curve over $\F_5$ for $p=2,3,5$ and $e=1,3$ and raise a question. 
Secondly, we study conditions for $\lim_{n\to \infty} |{\rm Cl}^0(k_{E_{l^{p^n}}})| \in \Z$ with focus on the cases with $p=l$. 
Finally, we discuss the Iwasawa $\nu$-invariants for the cases with $\lim_{n\to \infty}|{\rm Cl}^0(k_{E_{l^{p^n}}})|=0$. 

\subsubsection{Observations} 
We examine the cases of $E:y^2=x^3+3x+3$, $l=5$, $e=1,3$, $p=2,3,5$. 
\begin{eg} \label{eg.q=5}
Let $l=5$ and $E:y^2=x^3+3x+3$. 
Then $|E(\F_5)|=F_{E_5}(1)=5$ and 
$F_{E_5}(t)=t^2-(5+1-5)t+5=t^2-t+5$. 

If $p=2$, then by $t^2-t+5\equiv (t^2+t+1)=\Phi_3(t)$ mod 2 and $F_{E_5}(-1)=5>0$,  
we have $\lim_{n\to \infty} |{\rm Cl}^0(E_{5^{2^n}})|=\Phi_3(1)=3$. 

If $p=3$, then by  $t^2-t+5\equiv t^2+2t+2$ mod 3, we have 
$\lim_{n\to \infty} |{\rm Cl}^0(E_{5^{3^n}})|=1-\frac{-1+\sqrt{-1}}{2}=\frac{3-\sqrt{-1}}{2}$, 
where $\zeta=\frac{-1+\sqrt{-1}}{2}$ is a primitive $8$-th root of unity 
with $\zeta+\ol{\zeta}\equiv -2$, 
$\zeta\ol{\zeta}\equiv 2$ mod 3. 

If $p=5$, then by $F_{E_5}(1)=5$,  we have $\lim_{n\to \infty} |{\rm Cl}^0(E_{5^{5^n}})|=0$ in $\Z_5$. 
Let $\alpha=\frac{1+\sqrt{-19}}{2}$ denote the larger root of $F_E(t)=t^2-t+5$, so that $|\alpha|_l$ holds. 
Since $\alpha-1$ is the smaller root of $E_F(1+t)=t^2+t+5$, we see $|\alpha-1|_5=5^{-1}<5^{-1/4}$.  
Thus, we have $\lambda=\nu=1$ 
and the value $|{\rm Cl}^0(E_{5^{5^n}})_{\text{non-}5}|=|{\rm Cl}^0(E_{5^{5^n}})|\, 5^{-(n+1)}$ converges to a non-zero value 
$\lim_{n\to \infty}{\rm Res}(t^{5^n}-1,t^2-t+5) 5^{-(n+1)}
=\lim_{n\to \infty} \frac{1-\alpha^{5^n}}{5^{n+1}}
=\frac{-\log \alpha}{5}=\frac{-1}{5}\log \frac{1+\sqrt{-19}}{2};$
\begin{center} 
\begin{tabular}{c|ccccccccc}
$n$&1&2&3&4&5&6&$\cdots$\\
\hline 
${\rm Res}(t^{5^n}-1,F_{E_5}(t))\,5^{-(n+1)}$
&$11^2$&{\scriptsize $11^2\times19704014845201$} &$\cdots$&&
{\scriptsize }&{\scriptsize }& $\cdots$\\
\hline 
mod $5^n$&1&21&71&321&1571&14071&$\cdots$
\end{tabular}.
\end{center} \ 
\end{eg}

\begin{eg} Let the notation be as in \Cref{eg.q=5} and put $q=5^3$ for instance. 
Then $F_{E_{5^3}}(t)=(t-\alpha^3)(t-\beta^3)=t^2+14t+125$, 
$F_{E_{5^{3p^n}}}(t)=(t- \alpha^{3p^n})(t- \beta^{3p^n})$, and hence  
$|{\rm Cl}^0(E_{5^{3p^n}})|=F_{E_{5^{3p^n}}}(1)={\rm Res}(t^{p^n}-1,t^2+14t+125)$. 
Note that $F_{E_{5^3}}(1)=140$. 

If $p=2$, then by $2\mid 140$, we have $\lim_{n\to \infty} |{\rm Cl}^0(E_{5^{3\cdot 2^n}})|=0$ in $\Z_2$. 
Since $t^2+14t+125\equiv t^2+1=(t-1)^2$ mod 2, we have $\lambda=2$ and 
$|{\rm Cl}^0(E_{5^{3\cdot 2^n}})|2^{-2n}$ converges to a non-zero value in $\Z_2$. 
Since $|\alpha^3-1|_2<1$, $|\beta^3-1|_2<1$, and $|(\alpha^3-1)(\beta^3-1)|_2=|140|_2=2^{-2}\geq (2^{-1})^2$, 
we may have $\nu>2$. 
In fact, we have $\nu=4$; 
If we put $\varepsilon=\alpha^3-1$ or $\beta^3-1$, then a direct calculation shows that $|\varepsilon|_2=|\varepsilon^2/2|_2=1/2$, 
$|\varepsilon-\varepsilon^2/2|_2=|\varepsilon^4/4|_2=1/4$, 
$|\varepsilon-\varepsilon^2/2-\varepsilon^4/4|_2=1/4$, while other terms satisfy $|\varepsilon^n/n|_2<1/4$. 
Thus, we have $|\log \alpha^3|_2=|\log \beta^3|_2=2^{-2}$, $\nu=4$, and 
$\lim_{n\to \infty}|{\rm Cl}^0(E_{5^{3\cdot 2^n}})_{\text{non-}2}|=2^{-4}(\log \alpha^3)(\log \beta^3)
=\frac{9}{16}(\log \alpha)(\log \beta)$, 
where $\log$ denotes the 2-adic logarithm extended to $\C_2$ so that $\log 2=0$; 
\begin{center} 
\begin{tabular}{c|cccccccccccc} 
$n$&0&1&2&3&4&5&6&7&8&9&10&$\cdots$\\ \hline 
${\rm Res}(t^{2^n}-1,F_{E_{5^3}}(t))2^{-2n-4}$ &35/4&7&245&{\scriptsize 953785}&
$\cdots$&&&&&&&\\
mod $2^n$&&1&1&1&1&17&17&17&145&401&401&$\cdots$\\
\end{tabular}. 
\end{center} 

If $p=3$, then 
by $3\nmid 140$ and $t^2+14t+125\equiv t^2+2t+2$ mod 3, the limit is the same as the case with $q=5$. 

If $p=5$, then by $5\mid 140$, 
we have $\lim_{n\to \infty} |{\rm Cl}^0(E_{5^{3\cdot 5^n}})|=0$. 
By $|\alpha^3-1|_5=|F_{E_{5^3}}(1)|=5^{-1}<5^{-1/4}$ and $|\beta^3-1|_5=1$, 
we have $\lambda=\nu=1$, and the value $|{\rm Cl}^0(E_{5^{3\cdot 5^n}})_{\text{non-}5}|=|{\rm Cl}^0(E_{5^{3\cdot 5^n}})|5^{-(n+1)}$ 
converges to a non-zero value 
$\frac{-\log \alpha^3}{5}=\frac{-3\log \alpha}{5}=\frac{-3}{5}\log \frac{1+\sqrt{-19}}{2}$. 
\end{eg}

The following question is an analogue of \Cref{q.knots} for elliptic curves; 
\begin{q} \label{q.elliptic} 
Consider the constant $\Z/n\Z$-extensions of the function field of an elliptic curve over $\F_l$. 

(1) Find all cases with $\lim=\lim_{n\to \infty} |{\rm Cl}^0(k_{E_{l^{p^n}}})| \in \Z$. 
Find conditions for the limits being specific values, especially for the case with $\lim=1$. 

(2) Find conditions of $(e,p)$ with $p\nmid e$ such that $\lim_{n\to \infty} |{\rm Cl}^0(k_{E_{l^{ep^n}}})|=0$ holds, and study the values of $\nu$. 
Can $\nu$ be arbitrarily large, while $|{\rm Cl}^0(k_{E_{l^{e}}})|$ being small? 
\end{q} 

\subsubsection{Cases with $\lim |{\rm Cl}^0(k_{E_{l^{p^n}}})| \in \Z$} \label{sss.EC.liminZ} 
First, we focus on the cases with $p=l$.

\begin{eg} Let $E:y^2=x^3-1$ with good reduction at $l\neq 2,3$. 
Let us investigate the $\Z_l$-extension of $k_E$. 
We have $F_E(t)=t^2-(l+1-|E(\F_l)|)t+l$. 
By Hasse's bound $|E(\F_l)|-(l+1)|\leq 2\sqrt{l}$, we have $l\mid (l+1-|E(\F_l)|)$ if and only if $l+1-|E(\F_l)|=0$. 
By \cite[Exercise 8.3.6]{DiamondShurman2005GTM}, we have $|E(\F_l)|\equiv 1$ if and only if $l\equiv 2$ mod 3. 
Thus, if $l\equiv 2$ mod 3, then both roots of $F_E(t)=t^2+l$ are smaller than 1, and hence $\lim_{n\to \infty}|{\rm Cl}^0(k_{E_{l^{l^n}}})|=1$. 
If instead $l\equiv 1$ mod 3, then the larger root $\alpha$ of $F_E(t)$ satisfies $|\alpha|_l=1$ and 
$\lim_{n\to \infty}|{\rm Cl}^0(k_{E_{l^{l^n}}})|=1-\zeta$ holds for the $l$-prime-th root of $\zeta$ with $|\alpha-\zeta|_l<1$. 
We have $\zeta=1$ if and only if 
$|E(\F_l)|\equiv 1-(-1)^{(l-1)/6}\binom{(l-1)/3}{(l-1)/2}\equiv 0$ mod $l$, which is in fact not the case by \cite{Olson1976JNT}. 
\end{eg} 

\begin{defn}[cf.\cite{Mazur1972Invent,LangTrotter1976LNM,DiamondShurman2005GTM}]
Let $E$ be an elliptic curve over $\Q$ with good reduction at $\F_l$. 
If $|E(\F_l)|\equiv 1$ mod $l$, then $l$ is called a \emph{supersingular prime} of $E$. 
If $|E(\F_l)|\equiv 0$ mod $l$, then $l$ is called an \emph{anomalous prime} of $E$. 
We say that $E$ has \emph{complex multiplication} (\emph{CM}) over $\Q(\sqrt{D})$ with $D<0$ 
if ${\rm End}(E)$ is isomorphic to an order of the ring of integers of $\Q(\sqrt{D})$. 
\end{defn} 

Put $a=1+l-|E(\F_l)|$. 
Since $F_E(t)=t^2-at+l\equiv t(t-a)$ mod $l$, 
we may write $F_E(t)=(t-\alpha)(t-\beta)$ with $|\alpha|_l\leq 1$ and $|\beta|_l< 1$. 
If $|\alpha|_l<1$, then we have $\lim_{n\to \infty} {\rm Res}(t^{l^n}-1,F_E(t)) =0$.
If $|\alpha|_l=1$, then we have $\lim_{n\to \infty} {\rm Res}(t^{l^n}-1,F_E(t)) =1-\xi$, 
where $\xi$ is the unique root of unity with order prime to $p$ satisfying $|\alpha-\xi|_l<1$.
We see that $\xi\in \Z$ if and only if $\xi=\pm1$. 
Therefore, if we have $\lim_{n\to \infty} {\rm Res}(t^{l^n}-1,F_E(t)) \in \Z$, 
then we have $\lim_{n\to \infty} {\rm Res}(t^{l^n}-1,F_E(t))=0,1,2$ in $\Z_l$ according as $t-a\equiv t-1,$ $t,$ $t+1$ mod $l$ and $|E(\F_l)|\equiv 0$, $1$, $2$ respectively. 
By Hasse's bound $|a|\leq 2\sqrt{l}$, we always have $F_E(1)\geq 0$ and $F_E(-1)\geq 0$. 
Hence these limits coincide with $\lim_{n\to \infty}|{\rm Cl}^0(k_{E_{l^{l^n}}})|$ for any $l$. 
Thus, we obtain the following.

\begin{prop} \label{prop.EC} 
Let $E$ be an elliptic curve over $\Q$ with good reduction at $\F_l$. 
If the limit $\lim_{n\to \infty} |{\rm Cl}^0(k_{E_{l^{l^n}}})|$ in $\Z_{l}$ is a rational integer, then  
$\lim_{n\to \infty} |{\rm Cl}^0(k_{E_{l^{l^n}}})|=0,1,2$. Moreover, 

$\bullet$ We have $\lim_{n\to \infty} |{\rm Cl}^0(k_{E_{l^{l^n}}})|=0$ in $\Z_l$ if and only if $l$ is an anomalous prime, that is, $|E(\F_l)|\equiv 0$ mod $l$ holds. We mostly have $\nu=1$ and the only exceptional cases are those with $l=2$, $|E(\F_l)|_2=1/2$, and $\nu=2$ {\rm (see Propositions \ref{prop.classify.elliptic} and \Cref{eg.p=l})}.  

$\bullet$ We have $\lim_{n\to \infty} |{\rm Cl}^0(k_{E_{l^{l^n}}})|=1$ in $\Z_l$ if and only if 
$l$ is a supersingular prime of $E$, that is, $|E(\F_l)|\equiv 1$ mod $l$ holds. 
If in addition $E$ has CM over $\Q(\sqrt{D})$ with $D<0$, 
then $\lim_{n\to \infty} |{\rm Cl}^0(k_{E_{l^{l^n}}})|=1$ in $\Z_l$ if and only if 
the Legendre symbol satisfies $(\frac{D}{l})\neq -1$. 

$\bullet$ We have $\lim_{n\to \infty} |{\rm Cl}^0(k_{E_{l^{l^n}}})|=2$ if and only if $|E(\F_l)|\equiv 2$ mod $l$ holds. 
\end{prop} 

Let $E$ be an elliptic curve over $\Q$ and put $a_l=1+l-|E(\F_l)|$. 
If $l\geq 5$, then the Hasse bound $|a_l|\leq 2\sqrt{l}$ yields the following table of $\lim=\lim_{n\to \infty} |{\rm Cl}^0(k_{E_{l^{l^n}}})|$. 
\begin{center}
\begin{tabular}{|c|c|c|c|} \hline 
$|E(\F_l)|$&$a_l$&$\lim$&$l$\\ \hline \hline 
$l$&$1$&$0$&anomalous\\  \hline 
$l+1$&0&1&supersingular\\ \hline 
$l+2$&$-1$&$2$&--\\ \hline 
\end{tabular} 
\end{center} 
Elkies \cite{Elkies1987Invent} proved that every $E$ has infinitely many supersingular primes. 
Mazur asked in \cite[Section 1, b)]{Mazur1972Invent} 
whether there are infinitely many anomalous primes, 
which is now proved in the affirmative for most cases \cite[Corollary 4.3]{BabinkostovaBahrKimNeymanTaylor2019JNT}. 
Primes with $a_l\equiv -1$ 
also appears in \cite[Remark 2.6]{BabinkostovaBahrKimNeymanTaylor2019JNT}. 
In addition, Namba--Sato and Serre--Tate's 
conjecture in the 1960s and Lang--Trotter's refinement expect the following: 
Let $N_E$ denote the conductor of $E$, let $r\in \Z$, and assume additionally $r\neq 0$ if $E$ is CM.
Then, $\pi_{E,r}(x)=\{l \mid l<x,\, a_l=r,\, l\nmid N_E\}$ $\sim C_{E,r}\sqrt{x}/\log x$ for a constant $C_{E,r}\geq 0$ (cf.~\cite{LangTrotter1976LNM,WanXi2021-arXiv}). 
Hence, by the prime number theorem $\pi(x)=\{l \mid l<x\}\sim x/\log x$ and the Hasse bound,  
the density of each of such primes in the set of all primes is expected to be zero, assuming $r\neq 0$ if $E$ is CM.

\begin{eg} Let $E:y^2=x^3-5$. Then $l=37$ is known to be 
an anomalous prime of $E$. 
We have $F_E(t)=t^2-t+37\equiv t(t-1)$ mod 37 and hence $\lim_{n\to \infty} |{\rm Cl}^0(k_{E_{l^{l^n}}})|=0$. 
Put $\alpha=\frac{1+\sqrt{-147}}{2}$ and $\beta=\frac{1-\sqrt{-147}}{2}$ so that we have $F_E(t)=(t-\alpha)(t-\beta)$ with $1/l=|\beta|_l<|\alpha|_l=1$. 
Since 
$\alpha-1$ is the smaller root of $F_E(1+t)$ $=t^2+t+37$, we have 
$|\alpha-1|_l=1/l<l^{-1/(l-1)}$ and hence $|\log \alpha|_l=1/l$. 
By \Cref{FoxWeberFF} and \Cref{thm.res},
we have $\lambda=\nu=1$ and 
the value $|{\rm Cl}^0(k_{E_{l^{l^n}}})_{\text{non-}l}|=|{\rm Cl}^0(k_{E_{l^{l^n}}})|\,l^{-(n+1)}$ converges to a non-zero value 
$\lim_{n\to \infty}{\rm Res}(t^{l^n}-1,F_E(t))l^{-(n+1)}$ 
$=\lim_{n\to \infty} \frac{1-\alpha^{l^n}}{l^{n+1}}
=\frac{-\log \alpha}{l}=\frac{-1}{37}\log\frac{1+\sqrt{-147}}{2};$ 
\begin{center} 
\begin{tabular}{c|ccccc}
$n$&0&1&2&3&$\cdots$\\
\hline  
${\rm Res}(t^{37^n}-1,F_E(t))37^{-(n+1)}$ mod $37^n$ &1&1&741&13062 & $\cdots$\\
\end{tabular}.
\end{center}
\end{eg}

Next, let us briefly examine the cases with $p\neq l$. 
Since $F_E(t)$ is monic, the argument becomes much easier than the case of twist knots in \Cref{prop.Atiyah}, yielding the following.  
\begin{prop} \label{prop.Atiyah.elliptic}
If $\lim=\lim |{\rm Cl}^0(k_{E_{l^{ep^n}}})| \in \Z$, then according as 
$F_E(t)=t^2-at+l\equiv (t+1)^2, (t-1)^2, (t+1)(t-1), t^2+t+1, t^2+1, t^2-t+1$ mod $p$, 
we have $\lim =4,0,0,3,2,1$ respectively. 
Especially, $\lim =1$ if and only if $l\equiv 1$ and $a=2-|E(\F_l)|\equiv -1$ mod $p$. 
\end{prop} 

Propositions \ref{prop.EC}, \ref{prop.Atiyah.elliptic} completes the list of the cases with the $p$-adic limit being in $\Z$. 
The cases with $\lim=0$ will be further discussed in below. 

\subsubsection{Can $\nu$ be large with $r_e$ being small?} \label{sss.ec.nu} 
For an elliptic curve $E$ over $\F_l$ with the function field $k$, 
let $k_n/k$ denote the constant $\Z/n\Z$-extension. 
Here we give a slightly systematic study of the Iwasawa $\nu$-invariants and answer the following question. 

\begin{q}[A paraphrase of \Cref{q.elliptic} (2)] 
 \label{q.nu.elliptic} 
For any $N>0$, find a $\Zp$-extension $k_{ep^n}/k_e$ with $p\nmid e$, $|{\rm Cl}^0(k_e)_{(p)}|<p^\nu$, $\nu>N$.  
\end{q} 

Put $F(t)=t^2-(l+1-|E(\F_l)|)t+l$, $a=l+1-|E(\F_l)|$. 
By Hasse's bound, we have $|a|\leq 2\sqrt{l}$. 
Put $r_n={\rm Res}(t^n-1,F(t))$, so that we have $|r_1|=|E(\F_l)|$, $|r_{n}|=|{\rm Cl}^0(k_n)|$. 
If we write $F(t)=(t-\alpha)(t-\beta)$, then we have $r_n=1+l^n-(\alpha^n+\beta^n)$. 
A similar argument to \Cref{prop.nu>0} yields the following. 

\begin{prop} \label{prop.nu>0.elliptic} 
For any $e\in \Z_{>0}$, the following conditions are equivalent. 
\begin{itemize}
\item $\lim_{n\to \infty}|{\rm Cl}^0(k_{ep^n})|=0$ in $\Zp$ 
\item $|{\rm Cl}^0(k_e)|\equiv 0$ mod $p$
\item $(k_{ep^n}/k_e)_n$ has $\nu>0$. 
\end{itemize}
Furthermore, except for the following special cases, we have $p^\nu=|{\rm Cl}^0(k_e)_{(p)}|$. 
\begin{itemize} 
\item $p=3$, $|r_e|_3=1/3$. 
\item $p=2$, $|r_e|_2=1/2,1/4$. 
\end{itemize} 
\end{prop}

Since $F(t)$ has less symmetricity than the $\Delta_K(t)$ of twist knots, 
we have slightly more exceptional cases than in \Cref{prop.nu>0}. 
More precisely, we have the following. 

\begin{prop} \label{prop.classify.elliptic} 
{\rm (1)} If $p\nmid l^e-1$, then $F_e(t)$ has just one root which is close to 1, and the only exceptional cases are $p=l=2$, $|r_e|_2=1/2$, $\nu=2$. 

{\rm (2)} If $p\mid l^e-1$, then $F_e(t)$ has two roots $\alpha^e$, $\beta^e$ close to 1. The only exceptional cases are $p=2$, $|r_e|_2=1/2$, $1/4$ and $p=3$, $|r_e|_3=1/3$. 

{\rm (i)} If in addition $|r_e+(l^e-1)|_p\leq |r_e|_p^{1/2}$, then we have $|\alpha^e-1|_p=|\beta^e-1|_p$. 

{\rm (ii)} If instead $|r_e+(l^e-1)|_p> |r_e|_p^{1/2}$, then then we have $|\alpha^e-1|_p\neq |\beta^e-1|_p$. 
\end{prop} 

\begin{proof} If we put $F_n(t)=(t-\alpha^n)(t-\beta^n)$, then we have $r_n=F_n(1)$, 
$F_n(t)=t^2-(\alpha^n+\beta^n)t+l^n=t^2-(l^n+1-r_n)t+l^n$, 
$F_n(t+1)=(t-(\alpha^n-1))(t-(\beta^n-1))=t^2+(2-\alpha^n-\beta^n)t+(\alpha^n-1)(\beta^n-1)=t^2+(r_n-(l^n-1))t+r_n$. 
Thus, whether $p\mid l^n-1$ or not determines the number of roots of $F_m(t)$ that are close to 1.

Suppose $p\mid l^e-1$, so that $|1-\alpha^e|_p<1$ and $|1-\beta^e|_p<1$. 
If $|1-\alpha^e|_p=|1-\beta^e|_p$, then 
$|r_e-(l^e-1)|_p=|(1-\alpha^e)+(1-\beta^e)|_p\leq |1-\alpha|_p=|r_e|^{1/2}$. 
If instead $|1-\alpha^e|_p>|1-\beta^e|_p$, then by $r_n=(1-\alpha^e)(1-\beta^e)$, we have 
$|r_e-(l^e-1)|_p=|(1-\alpha^e)+(1-\beta^e)|_p=|1-\alpha|_p>|r_e|^{1/2}$. 

Most part of the assertion is proved by a similar argument to \Cref{prop.nu>0}. 
\Cref{eg.p=l} completes the assertion (1). 
\Cref{eg.2-root.nugg0} completes the assertion (2). 
\end{proof} 

Let us study the exceptional cases in \Cref{prop.classify.elliptic} (1). 

\begin{eg} \label{eg.p=l} 
Suppose $p=l$ (so that $p\nmid l^e-1$). If $p=2$ and $|r_1|_2=1/2$, then we have $\nu=2$. 
If otherwise, we have $\nu=1$. 
\end{eg} 

\begin{proof} 
By $F_e(t+1)\equiv t^2-(l^n+1)t\not\equiv t^2-t$, $F_e(t)$ has just one root $\alpha^e$ such that $|r_1|=|\alpha^e-1|_p<1$. 
If $p>3$, then by $|r_e|_p\leq 1/p<1/p^{1/(p-1)}$, we have $|r_e|_p=|\alpha^e-1|_p=|\log \alpha|_p=p^{-\nu}$. 
If $p=2$ and $4\mid r_e$, then by $|r_e|_2\leq 1/4<1/2$, we have the same. 
Note that if $2\nmid e$, then $r_e/r_1$ is the square of an integer. Hence if $p=2$ and $|r_e|_2=1/2$, then we have $|r_e|_2=|r_1|_2=1/2$. 
If $p=l=2$ and $e=1$, then by Hasse's bound $|a|\leq 2\sqrt{2}$ and that $r_1=F(1)=3-a$, we have $|3-a|_1=1/2$, and hence $a=1$. 
In this case we have $F(t)=t^2-t+2$ and $\nu=2>1$. 
For instance, $E:y^2+xy-x^3-x=0$ is such a case. 
In addition, we have $F_3(t)=t^2+5t+8$ and $(r_{3^1 2^n})_n$ also have $\nu=2$. 
\end{proof} 

We next study exceptional cases in \Cref{prop.classify.elliptic} (2) (i)  with $e=1$. 
Note that $\lambda=2$. 

\begin{eg} \label{eg.2-root.nugg0} 
Let $p=2$ and $2\mid l-1$. \ 
%

$\bullet$ Let $|r_1|_2=|l+1-a|_2=1/2$. 
For instance, let $l=b 2^{1+c}+1$, $2\nmid b$, $b,c\in \Z_{\geq 0}$. 
Note that such a prime number exists for arbitrary large $c$ by Dirichlet's theorem on arithmetic progressions. 
If $a=0$, then we have $r_1=F(1)=2(b2^c+1)$, $r_4=2^{2c+4}b^2(b2^c+1)^2$, $v_2(r_4)=2c+4$. 
By $\lambda=2$, \Cref{prop.nu>0.elliptic} yields  $\nu=v_2(r_4)-2\cdot 2=2c$, while $|r_1|_2=1/2$. 

$\bullet$ Let $|r_1|_2=|l+1-a|_2=1/4$. 
For instance, if $l=b 2^{c+2}+1$, $2\nmid b$, $b,c\in \Z_{\geq 0}$, $a=-2$, then 
$v_2(r_1)=2$,  $v_2(r_4)=2c+6$, $\nu= 2c+6-4=2c+2$. 

$\bullet$ 
If $|r_1|_2=1/2^d$ with $d>2$, then $\nu$ is determined by $r_1$.
\end{eg}

\begin{eg} \label{eg.2-root.nugg0'} 
Let $p=3$. If $3\mid l-1$, 
then we have $|r_1|_3=|l+1-a|_3=1/3$.  
For instance, if $l=b 3^{2+c}+1$ with $3\nmid b$, $b,c\in\Z_{\geq 0}$ and $a=2$, then 
$v_3(r_{3^3})=c+8$, $\nu= c+8-6=c+2$. 
\end{eg}

Examples \ref{eg.2-root.nugg0} and \ref{eg.2-root.nugg0'} answer \Cref{q.nu.elliptic}. 
Cases in (2) (ii) may be treated similarly. 

\begin{rem} For any elliptic curves over $\F_l$, we have $\nu \geq 0$, as in the cases of knots. 
We wonder if there is an analogous situation to the cases of links with arbitrary negative $\nu$ (cf.~\!\Cref{rem.nu-links}). 
A systematic study of the Iwasawa $\nu$-invariants of number fields may be found in Sumida-Takahashi's works \cite{SumidaTakahashi2004JNT, SumidaTakahashi2007MC}. 
\end{rem} 

\begin{rem} We may study $\Zp$-covers of a 3-manifold and constant $\Zp$-extensions of a function field 
as subcovers or subextensions of the $\wh{\Z}$-cover or the $\wh{\Z}$-extension in a parallel manner. 
We expect further interactions between these objects.  
\end{rem}


\bibliographystyle{amsalpha}
\bibliography{UekiYoshizaki.pAdicLimit.arXiv.bbl}

\ \end{document}